\documentclass[11 pt]{article}

\usepackage{amsmath}
\usepackage{amsthm}
\usepackage{amssymb}
\usepackage{mathrsfs}
\usepackage{xcolor}
\usepackage[normalem]{ulem}
\usepackage{mathtools}
\usepackage{enumitem}

\usepackage[utf8]{inputenc}
\usepackage[T1]{fontenc}
\usepackage{lmodern}

\usepackage[backend=bibtex,style=alphabetic, url=false,doi=false,maxbibnames=50,maxcitenames=50,maxnames=50]{biblatex}
\defbibheading{bibliography}[Bibliography]{\section*{#1}\markboth{#1}{#1}}

\usepackage{titlesec}
\titleformat{\section}
  {\Large\center\bfseries\scshape}
  {\thesection.}{.7em}{}
\titlespacing*{\section}{0pt}{3.5ex plus 0ex minus 0ex}{1.5ex plus 0ex}
\titleformat{\subsection}
  {\center\bfseries\scshape}
  {\thesubsection.}{.7em}{}
\titlespacing*{\subsection}{0pt}{3.5ex plus 0ex minus 0ex}{1.5ex plus 0ex}

\usepackage{hyperref}
\hypersetup{citecolor = black,colorlinks,linkcolor = black,urlcolor = gray}

\usepackage[capitalize]{cleveref}

\newcommand{\Cech}{\v{C}ech}
\newcommand{\Cesaro}{Ces\`{a}ro}
\newcommand{\Erdos}{Erd\H{o}s}
\newcommand{\Folner}{F\o{}lner}

\newcommand{\Szemeredi}{Szemer\'{e}di}
\newcommand{\Beiglbock}{Beiglb{\"o}ck}

\newcommand{\compact}{\comp}
\newcommand{\unitary}{U}
\newcommand{\comp}{\mathsf{c}}
\newcommand{\wm}{\mathsf{wm}}
\newcommand{\bes}{\mathsf{Bes}}
\newcommand{\anti}{\mathsf{anti}}

\newcommand{\bilin}[2]{\langle {#1},\ {#2} \rangle}
\newcommand{\bilinsq}[2]{\left[ {#1},\ {#2} \right]}
\newcommand{\nbar}{\Vert}
\newcommand{\lp}[1]{\mathsf{L}\!^{\mathsf{#1}}}
\newcommand{\N}{\mathbb{N}}
\newcommand{\Z}{\mathbb{Z}}

\newcommand{\R}{\mathbb{R}}
\newcommand{\C}{\mathbb{C}}
\newcommand{\T}{\mathbb{T}}

\newcommand{\dens}{\mathsf{d}}

\newcommand{\upperdens}{\overline{\dens}}
\newcommand{\ultra}[1]{\mathsf{#1}}

\newcommand{\rmult}{\mathsf{R}}

\newcommand{\tmult}{\mathsf{T}}
\newcommand{\smult}{\mathsf{S}}
\newcommand{\umult}{\mathsf{U}}
\newcommand{\cont}{\mathsf{C}}

\newcommand{\ess}{\mathsf{Ess}}
\newcommand{\haar}{\mathsf{m}}

\newcommand{\define}[1]{\textbf{#1}}
\newcommand{\cl}{\mathsf{cl}}

\newcommand{\M}[1]{\mathcal{M}(#1)}

\newcommand{\Hilb}{\mathscr{H}}

\renewcommand{\d}{\,\mathsf{d}}
\renewcommand{\setminus}{\backslash}

\newtheoremstyle{plain}{3mm}{3mm}{\slshape}{}{\bfseries}{.}{.5em}{}
\newtheorem{theorem}{Theorem}[section]
\newtheorem{conjecture}[theorem]{Conjecture}
\newtheorem{lemma}[theorem]{Lemma}
\newtheorem{proposition}[theorem]{Proposition}
\newtheorem{corollary}[theorem]{Corollary}
\newtheoremstyle{definition}{2mm}{2mm}{}{}{\bfseries}{.}{.5em}{}
\theoremstyle{definition}
\newtheorem{definition}[theorem]{Definition}
\newtheorem{remark}[theorem]{Remark}
\newtheorem{question}[theorem]{Question}
\newtheorem{example}[theorem]{Example}

\title{A proof of a sumset conjecture of \Erdos{}}
\author{Joel Moreira\thanks{\textsc{Northwestern University}~--~\href{mailto:joel.moreira@northwestern.edu}
{\texttt{joel.moreira@northwestern.edu}}} \and Florian K.\ Richter\thanks{\textsc{Northwestern University}~--~\href{mailto:fkr@northwestern.edu}
{\texttt{fkr@northwestern.edu}}} \and Donald Robertson\thanks{\textsc{University of Utah}~--~\href{mailto:robertso@math.utah.edu}
{\texttt{robertso@math.utah.edu}}}}

\allowdisplaybreaks

\definecolor{dg}{RGB}{0,128,0}

\begin{document}

\maketitle
\small

\begin{abstract}
In this paper we show that every set $A \subset \mathbb{N}$ with positive density
contains $B+C$ for some pair $B,C$ of infinite subsets of $\mathbb{N}$, settling a conjecture of \Erdos{}.
The proof features two different decompositions of an arbitrary bounded sequence into a structured component and a pseudo-random component.
Our methods are quite general, allowing us to prove a version of this conjecture for countable amenable groups.
\end{abstract}

\tableofcontents

\thispagestyle{empty}
\normalsize

\section{Introduction}
\label{sec_intro}

\paragraph{History and previous results.}

Sumsets $B + C \coloneqq \{ b + c : b \in B, c \in C \}$ for $B,C \subset \N$ are a central object of study in additive combinatorics.
In particular, it is natural to ask which sets $A\subset\N$ contain a sumset $B + C$ with $B$ and $C$ infinite.
It follows from the infinite version of  Ramsey's theorem \cite[Theorem~A]{MR1576401} that, whenever $\N$ is finitely partitioned, one of the cells contains $B+C$ for infinite sets $B,C \subset \N$; this is also an immediate corollary of Hindman's theorem~\cite{Hindman79}.
The following conjectured density analogue, attributed to \Erdos{} in \cite{Nathanson80}, is called an ``old problem'' in \cite[85]{Erdos_Graham80}.

\begin{conjecture}[\Erdos{} sumset conjecture]
\label{conjecture_erdossumset}
If $A\subset\N$ has positive upper density, i.e.
\[
\limsup_{N\to\infty} \frac{|A\cap \{1,\ldots,N\}|}{N} > 0,
\]
then $A$ contains $B + C $, where $B$ and $C$ are infinite subsets of $\N$.
\end{conjecture}

Nathanson \cite{Nathanson80} showed that a set $A$ with positive upper density contains a sum $B+C$ for a set $B$ of positive density and a set $C$ of any finite cardinality.
More recently, Di Nasso, Goldbring, Jin, Leth, Lupini and Mahlburg~\cite{DGJLLM15} employed non-standard analysis and ideas from ergodic theory to show that a set $A\subset\N$ with upper density greater than $1/2$ contains a sum $B+C$ where $B$ and $C$ are infinite sets.
As a corollary, derived using Ramsey's theorem and a result of Hindman \cite[Theorem 3.8]{MR0677568}, it follows that if $A$ has positive upper density, then for some $t\in\N$ the union $A\cup(A-t)$ contains a sum $B+C$ where $B$ and $C$ are infinite sets.
Some further progress on a variant of \cref{conjecture_erdossumset} was also made in \cite{arXiv:1701.07791}.

\paragraph{Main results.}

The goal of this paper is to verify \cref{conjecture_erdossumset}.
In fact we prove a stronger result.
Recall that a \define{\Folner{} sequence} in $\N$ is any sequence $\Phi \colon N \mapsto \Phi_N$ of finite, non-empty subsets of $\N$ satisfying
\[
\lim_{N\to\infty}\frac{\big|(\Phi_N + m) \triangle \Phi_N\big|}{|\Phi_N|} = 0
\]
for all $m \in \N$.
For example, any sequence $N \mapsto\{a_N+1,a_N+2,\dots,b_N\}$ of intervals in $\N$ with length $b_N-a_N$ tending to infinity is a \Folner{} sequence.
Given a \Folner{} sequence $\Phi$ and a set $A\subset\N$ the quantity
\[
\upperdens_\Phi(A)\coloneqq\limsup_{N\to\infty}\frac{|A\cap \Phi_N|}{|\Phi_N|}
\]
is the \define{upper density} of $A$ with respect to $\Phi$.
If
\[
\lim_{N \to \infty} \frac{|A\cap \Phi_N|}{|\Phi_N|}
\]
exists we denote it by $\dens_\Phi(A)$ and call it the \define{density} of $A$ with respect to $\Phi$.
The following is our main result, which verifies a generalization of \cref{conjecture_erdossumset} to \Folner{} sequences.

\begin{theorem}
\label{thm_erdosNatural}
For every $A \subset \N$ that satisfies $\upperdens_\Phi(A)>0$ for some \Folner{} sequence $\Phi$ one can find infinite sets $B,C \subset \N$ with $B+C \subset A$.
\end{theorem}

In fact, our methods are flexible enough to prove a version of \cref{thm_erdosNatural} in countable amenable groups.
A \define{two-sided \Folner{} sequence} on a discrete countable group $G$ is any sequence $\Phi\colon N \mapsto \Phi_N$ of finite, non-empty subsets of $G$ satisfying
\begin{equation}
\label{eq_folner}
\lim_{N \to \infty} \frac{|(\Phi_N g)\, \triangle \, \Phi_N|}{|\Phi_N|}
= 0 =
\lim_{N \to \infty} \frac{|\Phi_N\, \triangle \, (g \Phi_N)|}{|\Phi_N|}
\end{equation}
for all $g \in G$.
A countable group $G$ is called \define{amenable} if and only if it admits a two-sided \Folner{} sequence (cf.\ \cite{Greenleaf69,Tomkowicz_Wagon16}).
Given a two-sided \Folner{} sequence $\Phi$ on $G$ and a set $A\subset G$, the quantity
\begin{equation}
\label{eqn_folnerDensity}
\upperdens_\Phi(A)\coloneqq\limsup_{N \to \infty} \frac{|A \cap \Phi_N|}{|\Phi_N|}
\end{equation}
is the \define{upper density} of $A$ with respect to $\Phi$.
If
\[
\lim_{N \to \infty} \frac{|A \cap \Phi_N|}{|\Phi_N|}
\]
exists then we denote it by $\dens_\Phi(A)$ and call it the \define{density} of $A$ with respect to $\Phi$.

\begin{theorem}
\label{thm_erdos_sumset_amenable}
Let $G$ be a countable group, let $\Phi$ be a two-sided \Folner{} sequence on $G$ and let $A \subset G$ be such that $\upperdens_\Phi(A) > 0$.
Then there are infinite sets $B,C \subset G$ with $BC= \{ bc : b \in B, c \in C \} \subset A$.
\end{theorem}

\paragraph{Strategy of the proof.}
We outline here quite broadly the main ideas in the proof of \cref{thm_erdosNatural}.
We freely make use of terminology that is only defined later in the paper.
In particular, the relevant background on ultrafilters is given at the beginning of \cref{sec_outlineNatural}.

To begin with, we borrow ideas from \cite{DGJLLM15} to show that whenever one has
\begin{equation}
\label{eqn:strategy}
\lim_{m \to \ultra{p}} \dens_\Psi((A-m) \cap (A-\ultra{p}))
>0
\end{equation}
for some \Folner{} sequence $\Psi$ and some non-principal ultrafilter $\ultra{p}$, necessarily $A$ contains a sum $B+C$ with $B,C\subset\N$ infinite.
Here we write $A-\ultra{p}$ for the set $\{n \in \N : A-n \in \ultra{p}\}$.
Thus the main part of our proof of \cref{thm_erdosNatural} consists of finding, for every \Folner{} sequence $\Phi$ and every $A\subset\N$ with $\upperdens_\Phi(A)>0$, a non-principal ultrafilter $\ultra{p}$ and a \Folner{} subsequence $\Psi$ of $\Phi$ such that \eqref{eqn:strategy} is satisfied.

Given $f\colon\N\to\C$ and $m \in \N$, write $\rmult^m f$ for the function $n\mapsto f(m+n)$.
If in addition $\ultra{p}$ is an ultrafilter on $\N$ we write $\rmult^\ultra{p} f$ for the function
\[
n\mapsto \lim_{m\to\ultra{p}}f(n+m)
\]
for all $n \in \N$.
In doing so one can rewrite $1_{A - m}$ as $\rmult^m 1_A$ and $1_{A - \ultra{p}}$ as $\rmult^\ultra{p} 1_A$.
We can therefore rewrite \eqref{eqn:strategy} in the form
\begin{equation}
\label{eqn:strategyFunctional}
\lim_{m \to \ultra{p}}\bilin{\rmult^m 1_A}{\rmult^{\ultra{p}}1_A}_{\Psi}
>
0
\end{equation}
where, for two bounded functions $f,g\colon\N\to\C$, the inner product $\bilin{ \cdot}{\cdot}_{\Psi}$ is defined as
$$
\bilin{ f}{g}_{\Psi}\coloneqq  \lim_{N \to \infty} \frac{1}{|\Psi_N|} \sum_{n \in \Psi_N}  f(n) \, \overline{g(n)}.
$$

The utility of ultrafilters in our proof is two-fold.
On the one hand, the language of ultrafilters leads us to \eqref{eqn:strategy} and \eqref{eqn:strategyFunctional}, which are similar to expressions encountered in other problems of additive combinatorics.
In fact, having reduced the proof of \cref{thm_erdosNatural} to a statement involving the bilinear functional $(f,g)\mapsto\lim_{m \to \ultra{p}}\bilin{\rmult^m f}{\rmult^{\ultra{p}}g}_{\Psi}$ is particularly useful, since it opens the door for using tools and ideas from functional analysis and ergodic Ramsey theory.
On the other hand, shifts by ultrafilters are more versatile than shifts by natural numbers, which we exploit at numerous different places in the proof of \cref{thm_erdosNatural}.

In \cite[Theorem~5.5]{DGJLLM15} the language of non-standard analysis was used to verify \eqref{eqn:strategyFunctional} when $A$ is ``pseudo-random''.
Roughly speaking, the set $A$ is pseudo-random if it is almost independent from most of its shifts.
It is natural to ask~\cite[Questions~5.6, 5.7]{DGJLLM15} what happens when $A$ is not pseudo-random.
In this case, it is beneficial to employ a decomposition of $1_A$ into structured and pseudo-random components.
Inspired by the Jacobs--de Leeuw--Glicksberg splitting on Hilbert spaces \cite{MR0077092,MR0131784}, we prove that $1_A$ can always be decomposed as a sum $f_\wm + f_\comp$ of a weak mixing function $f_\wm$ and a compact function $f_\comp$.
We think of $f_\wm$ as being the ``pseudo-random'' component of $1_A$ and of $f_\comp$ as the ``structured'' component of $1_A$.

The decomposition $1_A = f_\wm + f_\comp$ is stable under shifts by $m \in \N$ in the sense that $\rmult^m f_\wm + \rmult^m f_\comp$ is the decomposition of $\rmult^m 1_A = 1_{A-m}$ into weak mixing and compact functions.
In light of this fact, we can consider the left hand side of \eqref{eqn:strategyFunctional} as a sum of two terms, one with $\rmult^m 1_A$ replaced by the weak mixing function $\rmult^m f_\wm$, the other with $\rmult^m 1_A$ replaced by the compact function $ \rmult^m f_\comp$:
\begin{equation}
\label{eqn:first_splitting}
\lim_{m \to \ultra{p}}\bilin{ \rmult^m 1_A}{ \rmult^{\ultra{p}} 1_A}_{\Psi,\ultra{p}}= \lim_{m \to \ultra{p}}\bilin{ \rmult^m f_\wm}{ \rmult^{\ultra{p}} 1_A}_{\Psi}+ \lim_{m \to \ultra{p}}\bilin{ \rmult^m f_\comp}{  \rmult^{\ultra{p}} 1_A}_{\Psi}.
\end{equation}

Unfortunately, the decomposition into compact and weak mixing components is not stable under shifts by ultrafilters, so we are unable to use it to understand $\rmult^\ultra{p}1_{A}$.
For this reason we devise a second splitting whose interaction with ultrafilters we are able to control.
This second splitting asserts that $1_A = f_\anti + f_\bes$, where the ``structured'' component $f_\bes$ is a Besicovitch almost periodic function, which is a stronger property then being a compact function, and the complement $f_\anti$ is characterized by being orthogonal to $e^{2\pi i n \theta}$ for all $\theta\in[0,1)$, which is a weaker form of ``pseudo-randomness'' than weak mixing.
It is the specialized nature of $f_\bes$ that reacts well with ultrafilters.

Applying our second splitting to $\rmult^{\ultra{p}}1_{A}$ in the last term of \eqref{eqn:first_splitting} leaves us with a sum of the following three terms.
\begin{align}
\label{eqn:strategyFunctional_-1}
&
\lim_{m \to \ultra{p}}\bilin{ \rmult^m f_\comp}{ \rmult^{\ultra{p}} f_\bes}_{\Psi}
\\
\label{eqn:strategyFunctional_-2}
&
\lim_{m \to \ultra{p}}\bilin{ \rmult^m f_\comp}{ \rmult^{\ultra{p}} f_\anti}_{\Psi}
\\
\label{eqn:strategyFunctional_-3}
&
\lim_{m \to \ultra{p}}\bilin{ \rmult^m f_\wm}{ \rmult^{\ultra{p}}1_A}_{\Psi}
\end{align}
We show that \eqref{eqn:strategyFunctional_-3} is zero using the pseudo-randomness of weak mixing.
Positivity of the term \eqref{eqn:strategyFunctional_-1} follows from the close relationship between $f_\bes$ and its shifts by ultrafilters.
The remaining term, \eqref{eqn:strategyFunctional_-2}, which involves $f_\comp$ and $f_\anti$, is the most delicate.
To show it is non-negative we adapt an argument of \Beiglbock{}~\cite{Beiglbock11}.
All together, this proves that the sum of the three terms in \eqref{eqn:strategyFunctional_-1}, \eqref{eqn:strategyFunctional_-2}, and \eqref{eqn:strategyFunctional_-3} is positive, which implies \eqref{eqn:strategyFunctional}.

It is reasonable to ask why we do not apply the splitting $f_\bes + f_\anti$ to both occurrences of $1_{A}$ in \eqref{eqn:strategyFunctional}.
The reason lies in the strength of the pseudo-randomness that weak mixing provides.
We would not be able to handle the hypothetical term
\[
\lim_{m \to \ultra{p}}\bilin{ \rmult^m f_\anti}{\rmult^{\ultra{p}} 1_A}_{\Psi}
\]
pairing $f_\anti$ with $1_A$, whereas we are able to handle \eqref{eqn:strategyFunctional_-3}.

\paragraph{Structure of the paper.}
The purpose of \cref{sec_outlineNatural} is to review the relevant material on ultrafilters and then to prove that \eqref{eqn:strategy} implies \cref{thm_erdosNatural}.
In \cref{sec_two_decomps} we prove our two splitting results.
The proof of \cref{thm_erdosNatural} is concluded in  \cref{sec_finding_p}.
In \cref{sec:erdosGroups} we explain the few steps where the proof of \cref{thm_erdos_sumset_amenable} differs from that of \cref{thm_erdosNatural}.
Finally, in \cref{sec_questions} we discuss some relevant open questions.

\paragraph{Acknowledgements.}
The first author is grateful for the support of the NSF via grant DMS-1700147.
The third author is grateful for the support of the NSF via grant DMS-1703597.
We thank John H.\ Johnson for providing useful references and Vitaly Bergelson for reading and commenting on an early draft of this paper.
We would also like to thank the anonymous referees for their careful reading of the manuscript and their detailed suggestions.
Their efforts have improved the readability of the paper.
We thank Bernard Host and Bryna Kra for pointing out a mistake in an earlier version of the proof of \cref{thm_jdlgSplitting}.

The first author became interested in the \Erdos{} sumset conjecture while visiting Martino Lupini at CalTech in June 2017 and is thankful for the hospitality provided.
We are also thankful to Michael Bj\"orklund, Alexander Fish and Anush Tserunyan for interesting discussions and to AIM for hosting the workshop ``Nonstandard methods in combinatorial number theory'' in August 2017 at which these discussions took place.

\section{Ultrafilter reformulation}
\label{sec_outlineNatural}

For the proofs of \cref{thm_erdosNatural} and \cref{thm_erdos_sumset_amenable} we found it crucial to rely on the theory of ultrafilters, which has proven to be very effective in solving problems in Ramsey theory in the past.
In this section we recall briefly some of the basic definitions and facts that we will utilize in this paper and then reduce \cref{thm_erdosNatural} to a statement of the form \eqref{eqn:strategy}.
Readers in want of a friendly introduction to ultrafilters may well enjoy \cite[Section 3]{Bergelson96}; for a comprehensive treatment see \cite{MR2893605}.

An \define{ultrafilter} on $\N$ is any non-empty collection $\ultra{p}$ of subsets of $\N$ that is closed under finite intersections and supersets and satisfies
\[
A\in\ultra{p}\iff \N \setminus A \notin\ultra{p}
\]
for every $A \subset \N$.
Given $n\in\N$, the collection $\ultra{p}_n\coloneqq \{A \subset \N : n \in A \}$ is an ultrafilter; ultrafilters of this kind are called \define{principal}.
We embed $\N$ in $\beta \N$ using the map $n \mapsto \ultra{p}_n$.
For the existence of non-principal ultrafilters, which follows from the axiom of choice, see \cite[Theorem~3.8]{MR2893605}.

The set of all ultrafilters on $\N$ is denoted by $\beta \N$.
Given $A \subset \N$ and using the above embedding of $\N$ in $\beta \N$, write $\cl(A) \coloneqq \{ \ultra{p} \in \beta \N : A \in \ultra{p} \}$ for the \define{closure} of $A$ in $\beta \N$.
The family $\{ \cl(A) : A \subset \N \}$ forms a base for a topology on $\beta \N$ with respect to which $\beta \N$ is a compact Hausdorff space.
We note that $\cl(A) \cap \cl(B) = \cl(A \cap B)$ for all $A,B \subset \N$.
The map $n \mapsto \ultra{p}_n$ embeds $\N$ densely in $\beta\N$.
Endowed with this topology, $\beta\N$ can be identified with the Stone--\Cech{} compactification of $\N$, which means that it has the following universal property:
for any function $f \colon \N \to K$ into a compact Hausdorff space $K$ there is a unique continuous function $\beta f \colon \beta\N \to K$ such that $(\beta f)(\ultra{p}_n)=f(n)$ for all $n \in \N$.
When no confusion may arise we denote $\ultra{p}_n$ simply by $n$.

Given a function $f\colon\N\to K$ with $K$ a compact Hausdorff space and given an ultrafilter $\ultra{p}\in\beta\N$, one can characterize $(\beta f)(\ultra{p})$ as the unique point $x$ in $K$ such that, for any neighborhood $U$ of $x$, the set $\{n\in\N : f(n) \in U\}$ belongs to $\ultra{p}$.
For this reason we use the notation
\[
\lim_{n\to\ultra{p}}f(n)\coloneqq (\beta f)(\ultra{p}).
\]

Given a set $A\subset\N$ we define
\[
A-\ultra{p}\coloneqq\{n\in\N:A-n\in\ultra{p}\}
\]
for all ultrafilters $\ultra{p}$ on $\N$.
Addition on $\N$ can be extended to a binary operation $+$ on $\beta \N$ by
\[
\ultra{p} + \ultra{q} = \{ A \subset \N : A - \ultra{q} \in \ultra{p} \}=\lim_{n\to\ultra{p}}\lim_{m\to\ultra{q}}n+m
\]
for all $\ultra{p},\ultra{q}$ in $\beta \N$.
We remark that despite being represented with the symbol $+$, this operation is \emph{not} commutative.
We mention this operation only to present the following lemma giving a criterion for a set of natural numbers to contain $B+C$; it will not be used throughout in the proof of \cref{thm_erdosNatural}.
This lemma was independently discovered by Di Nasso and a proof was presented in \cite[Proposition 3.1]{arXiv:1701.07791}.

\begin{lemma}[cf.\ \cref{lem:ultrafilterErdosGroups}]
\label{lem:ultrafilterErdos}
Fix $A \subset \N$.
There are non-principal ultrafilters $\ultra{p}$ and $\ultra{q}$ with the property that $A \in \ultra{p} +\ultra{q}$ and $A \in \ultra{q} + \ultra{p}$ if and only if there are infinite sets $B,C \subset \N$ with $B+C \subset A$.
\end{lemma}

Here is the main theorem of this section, which is inspired by the proof of \cite[Theorem~3.2]{DGJLLM15}.

\begin{theorem}
\label{prop_withultrafilters}
Let $A\subset\N$.
If there exist a \Folner{} sequence $\Phi$ in $\N$ and a non-principal ultrafilter $\ultra{p}\in\beta\N$ such that
$\dens_\Phi\big((A-n)\cap(A-\ultra{p})\big)$ exists for all $n\in\N$ and
\begin{equation}
\label{eq_prop_withultrafilters}
\lim_{n\to\ultra{p}}\dens_\Phi\big((A-n)\cap(A-\ultra{p})\big)>0
\end{equation}
then there exist infinite sets $B,C\subset\N$ such that $A\supset B+C$.
\end{theorem}

The following result of Bergelson~\cite{Bergelson85} will be crucial for the proof of \cref{prop_withultrafilters}.
We present a short proof of it for completeness.
\begin{lemma}
[{cf.\ \cite[Theorem 1.1]{Bergelson85}}]\label{thm_bergelson}
Let $(X,\mathcal{B},\mu)$ be a probability space and let $n \mapsto B_n$ be a sequence in $\mathcal{B}$.
Assume that there exists $\epsilon>0$ such that $\mu(B_n)\geq\epsilon$ for all $n\in\N$.
Then there exists an injective map $\sigma \colon \N \to \N$ such that
\begin{equation}\label{eq_thm_bergelson}
\mu\left(B_{\sigma(1)}\cap\cdots\cap B_{\sigma(n)}\right)>0
\end{equation}
for every $n \in \N$.
\end{lemma}
\begin{proof}
The collection ${\mathcal F}$ of all finite sets $F\subset\N$ with the property that $\mu(\bigcap_{n\in F}B_n)=0$ is countable, and therefore the union $X_0=\bigcup_{F\in{\mathcal F}}(\bigcap_{n\in F}B_n)$ has $\mu(X_0)=0$.

For each $N\in\N$ let $f_N\coloneqq\frac1N\sum_{n=1}^N1_{B_n}$. It is clear that $\int_X f_N\d\mu\geq\epsilon$ for every $N\in\N$.
By Fatou's lemma, the function $f:=\limsup_{N\to\infty}f_N$ also satisfies $\int_Xf\d\mu\geq\epsilon$.
Therefore there exists a point $x\in X\setminus X_0$ with $f(x)>0$, and in particular the set $\{n\in\N:x\in B_n\}$ is infinite.
Let $\sigma(n)$ be an enumeration of that set.

To show that \eqref{eq_thm_bergelson} holds notice that, for every $n\in\N$, the set $\{\sigma(1),\ldots,\allowbreak \sigma(n)\}$ can not be in ${\mathcal F}$ because $x\in B_{\sigma(1)}\cap\ldots\cap B_{\sigma(n)}$ but $x\notin X_0$.
\end{proof}

Given a \Folner{} sequence $\Phi$ on $\N$ write $\M{\Phi}$ for the set of Radon probability measures on $\beta \N$ that are weak$^*$ accumulation points of the set $\big\{\mu_N:N\in\N\big\}$, where
\begin{equation}
\label{eqn_measures_on_betaN}
\mu_N\coloneqq \frac{1}{|\Phi_N|} \sum_{n \in \Phi_N} \delta_n
\end{equation}
and $\delta_n$ is the unit mass at the principal ultrafilter $\ultra{p}_n$.
\begin{corollary}
\label{cor_bergelsonlemma}
Let $\Phi$ be a \Folner{} sequence on $\N$ and, for each $n\in\N$, let $A_n\subset\N$.
Assume $\dens_\Phi(A_n)$ exists for all $n \in \N$ and that there exists $\epsilon>0$ such that $\dens_\Phi(A_n)\geq\epsilon$ for all $n\in\N$.
Then there exists an injective sequence $\sigma\colon\N\to\N$ such that
\[
\upperdens_\Phi\left(A_{\sigma(1)}\cap\cdots\cap A_{\sigma(n)}\right)>0
\]
for every $n \in \N$.
\end{corollary}
\begin{proof}
Let $\mu\in\M\Phi$ and let $B_n=\cl(A_n)$.
The set $B_n$ is clopen and the density of $A_n$ along $\Phi$ exists so $\mu(B_n)=\dens_\Phi(A_n)$ for all $n \in \N$.
Apply \cref{thm_bergelson} to the probability space $(\beta\N,\mathcal{B},\mu)$, where $\mathcal{B}$ is the Borel $\sigma$-algebra on $\beta\N$, to find an injective map $\sigma\colon\N\to\N$ such that \eqref{eq_thm_bergelson} holds for every $n\in\N$.
Since $B_{\sigma(1)}\cap\cdots\cap B_{\sigma(n)}=\cl(A_{\sigma(1)}\cap\cdots\cap A_{\sigma(n)})$, this implies that $\upperdens_\Phi\left(A_{\sigma(1)}\cap\cdots\cap A_{\sigma(n)}\right)\geq\mu\big(B_{\sigma(1)}\cap\cdots\cap B_{\sigma(n)}\big)>0$ as desired.
\end{proof}

The next proposition, whose statement (and proof) is heavily influenced by the paper \cite{DGJLLM15}, can be seen as an ultrafilter-free version of \cref{prop_withultrafilters}.

\begin{proposition}
\label{prop_depoint}
Let $A\subset\N$.
If there exist a \Folner{} sequence $\Phi$ in $\N$, a set $L\subset\N$ and $\epsilon>0$ such that $\dens_\Phi\big((A-m)\cap L\big)$ exists for every $m\in\N$, and for every finite subset $F\subset L$
\begin{equation}
\label{eq_prop_depoint}
\bigcap_{\ell\in F}(A-\ell)\ \cap\ \Big\{m\in\N:\dens_\Phi\big((A-m)\cap L\big)>\epsilon\Big\}\text{ is infinite}
\end{equation}
then there exist infinite sets $B,C$ such that $A\supset B+C$.
\end{proposition}

\begin{proof}
Let $F_1 \subset F_2 \subset \cdots$ be an increasing exhaustion of $L$ by finite subsets.
Construct a sequence $n \mapsto e_n$ in $\N$ of distinct elements such that
\[
e_n \in \bigcap_{\ell \in F_n} (A-\ell)\cap\Big\{ m \in \N : \dens_\Phi\big((A-m) \cap L\big) > \epsilon \Big\}
\]
for each $n \in \N$.
This can be done because each of the sets above is infinite by hypothesis.

In particular $\dens_\Phi\big((A-e_n) \cap L\big) > \epsilon$ for all $n \in \N$.
The Bergelson intersectivity lemma (\cref{cor_bergelsonlemma}) implies that, for some subsequence $n \mapsto e_{\sigma(n)}$ of $e$ the intersection
\[
\Big( (A -e_{\sigma(1)}) \cap L \Big) \cap \cdots \cap \Big( (A-e_{\sigma(n)})\cap L \Big)
\]
is infinite for all $n \in \N$.

Choose $b_1 \in F_{\sigma(1)}$ and put $j_1 = 1$.
Choose $c_1 = e_{\sigma(1)}$.
Thus $c_1 \in A-b_1$.
Next choose $b_2 \in (A-c_1) \cap L$ outside $F_{\sigma(1)}$ and let $j_2$ be minimal with $b_2 \in F_{\sigma(j_2)}$.
(In particular $b_2$ is not equal to $b_1$.)
Then choose $c_2 = e_{\sigma(j_2)} \in (A-b_1) \cap(A-b_2)$.
Continue this process inductively, choosing
\[
b_{n+1}
\in
(A -c_1) \cap \cdots \cap (A- c_n) \cap L
=
(A -e_{\sigma(j_1)}) \cap \cdots \cap (A- e_{\sigma(j_n)}) \cap L
\]
outside $F_{\sigma(j_n)}$ and choosing $j_{n+1}$ minimal with $b_{n+1} \in F_{\sigma(j_{n+1})}$ and then choosing
\[
c_{n+1} = e_{\sigma(j_{n+1})} \in (A-b_1) \cap \cdots \cap (A-b_{n+1})
\]
which is distinct from $c_1,\dots,c_n$ because $e$ is injective.
Take $B = \{ b_n : n \in \N \}$ and $C = \{ c_n : n \in \N \}$ to conclude the proof.
\end{proof}

The proof of \cref{prop_withultrafilters} is now quite straightforward.

\begin{proof}[Proof of \cref{prop_withultrafilters}]
Let $L=A-\ultra{p}=\{\ell\in\N:A-\ell\in\ultra{p}\}$ and let $$\epsilon=\lim_{n\to\ultra{p}}\dens\big((A-n)\cap(A-\ultra{p})\big)/2.$$
Then the set $\{n\in\N:\dens\big((A-n)\cap L\big)>\epsilon\}$ is in $\ultra{p}$ and hence, for any finite set $F\subset L$, also the intersection
\[
\bigcap_{\ell\in F}(A-\ell)\ \cap\ \Big\{m\in\N:\dens_\Phi\big((A-m)\cap L\big)>\epsilon\Big\}
\]
is in $\ultra{p}$.
Since $\ultra{p}$ is non-principal, this intersection can not be finite.
The desired conclusion now follows from \cref{prop_depoint}.
\end{proof}

In view of \cref{prop_withultrafilters}, the proof of \cref{thm_erdosNatural} now follows from the following theorem.

\begin{theorem}
\label{thm_functionalErdosNatural}
Let $A\subset\N$ and let $\Phi$ be a \Folner{} sequence on $\N$ with $\dens_\Phi(A)$ existing.
For every $\epsilon>0$ there exists a \Folner{} subsequence $\Psi$ of $\Phi$ and a non-principal ultrafilter $\ultra{p} \in \beta \N$ such that $\dens_{\Psi}( (A-m) \cap (A - \ultra{p}) )$ exists for all $m\in \N$ and
\begin{equation}
\label{eqn_functionalErdosNatural}
\lim_{m \to \ultra{p}} \dens_{\Psi} \big( (A-m) \cap (A - \ultra{p}) \big)
\geq
\dens_{\Psi}(A)^2 - \epsilon
\end{equation}
holds.
\end{theorem}
\begin{proof}[Proof of \cref{thm_erdosNatural} assuming \cref{thm_functionalErdosNatural}]
Fix $A \subset \N$ with $\upperdens_\Phi(A) > 0$ for some \Folner{} sequence $\Phi$.
By passing to a subsequence of $\Phi$ we may assume that $\dens_\Phi(A)$ is defined and positive.
Apply Theorem~\ref{thm_functionalErdosNatural} with $\epsilon = \dens_\Phi(A)^2/2$.
Since $\dens_\Phi(A) = \dens_\Psi(A)$ for every further subsequence $\Psi$ of $\Phi$ the inequality \eqref{eqn_functionalErdosNatural} implies the hypothesis \eqref{eq_prop_withultrafilters} of Theorem~\ref{prop_withultrafilters}, so $A$ indeed contains $B+C$ for infinite sets $B,C \subset \N$.
\end{proof}

We conclude this section by reformulating \cref{thm_functionalErdosNatural} in a functional analytic language as in \eqref{eqn:strategyFunctional}.
Given a bounded function $f \colon \N \to \C$ define, for all $m\in\N$, the shift $\rmult^m f \colon \N \to \C$ by
\[
(\rmult^m f)(n) \coloneqq f(n + m)
\]
for all $n \in \N$.
We extend this to all $\ultra{p}\in\beta\N$ by defining the function $\rmult^\ultra{p} f \colon \N \to \C$ by
\[
(\rmult^\ultra{p} f)(n) \coloneqq \lim_{m \to \ultra{p}} f(n + m)
\]
for all $n \in \N$.
Observe that $\rmult^\ultra{p_m} f=\rmult^m f$ for all principal ultrafilters $\ultra{p}_m$. Also, the indicator function of the set $A-\ultra{p}$ is the function $\rmult^\ultra{p}1_A$, where $1_A$ is the indicator function of $A$.

Given a \Folner{} sequence $\Phi$ in $\N$ and functions $f,h\colon \N \to \C$, define the \define{Besicovitch seminorm} of $f$ along $\Phi$ to be
\begin{equation}
\label{eqn:besSeminorm}
\nbar f \nbar_\Phi = \left( \limsup_{N \to \infty} \frac{1}{|\Phi_N|} \sum_{n \in \Phi_N} |f(n)|^2 \right)^{1/2}
\end{equation}
and the inner product
\[
\bilin{f}{h}_\Phi = \lim_{N \to \infty} \frac{1}{|\Phi_N|} \sum_{n \in \Phi_N} f(n) \overline{h(n)}
\]
whenever the limit exists.
Minkowski's inequality
\begin{equation}
\label{eqn:mink}
\left( \sum_{n \in \Phi_N} |f(n) + h(n)|^2 \right)^\frac{1}{2}
\le
\left( \sum_{n \in \Phi_N} |f(n)|^2 \right)^\frac{1}{2}
+
\left( \sum_{n \in \Phi_N} |h(n)|^2\right)^\frac{1}{2}
\end{equation}
implies that $\nbar f+h \nbar_\Phi\leq \|f\|_\Phi + \|h\|_\Phi$, and hence $\nbar \cdot \nbar_\Phi$ is indeed a seminorm on the set of functions $f : \N \to \C$ for which $\nbar f \nbar$ is finite.
The following facts will be used throughout the paper.
\begin{enumerate}
\item
If $\Psi$ eventually agrees with a subsequence of $\Phi$ then $\nbar f \nbar_\Psi \le \nbar f \nbar_\Phi$ for all $f : \N \to \C$;
\item
(Cauchy-Schwarz)
$|\bilin{f}{h}_\Phi| \le \nbar f \nbar_\Phi \nbar h \nbar_\Phi$ whenever $\bilin{f}{h}_\Phi$ exists and both $\nbar f \nbar_\Phi$, $\nbar h \nbar_\Phi$ are finite.
\item
If $\nbar f \nbar_\Phi$ is finite then there is a subsequence $\Psi$ of $\Phi$ such that $\nbar f \nbar_\Xi = \nbar f \nbar _\Phi$ for every subsequence $\Xi$ of $\Psi$.
\item
If $\nbar f \nbar_\Phi$ and $\nbar h \nbar_\Phi$ are both finite then there is a subsequence $\Psi$ of $\Phi$ such that $\bilin{f}{h}_\Psi$ exists.
\end{enumerate}

The following result, whose proof is given in \cref{sec_finding_p} using the material of \cref{sec_two_decomps}, implies \cref{thm_functionalErdosNatural} by choosing $f=1_A$.

\begin{theorem}
\label{thm_hilbertErdosNatural}
Let $f$ be a non-negative bounded function on $\N$ and let $\Phi$ be a \Folner{} sequence on $\N$ such that $\bilin{1}{f}_\Phi$ exists.
For every $\epsilon > 0$ there exists a subsequence $\Psi$ of $\Phi$ and a non-principal ultrafilter $\ultra{p} \in \beta \N$ such that $\bilin{\rmult^m f}{\rmult^\ultra{p} f}_\Psi$ exists for all $m \in \N$ and
\begin{equation}
\label{eqn_hilbertErdosNatural}
\lim_{m \to \ultra{p}} \bilin{ \rmult^m f }{ \rmult^\ultra{p} f }_\Psi
\geq
\bilin{1}{f}_\Psi^2  - \epsilon
\end{equation}
holds.
\end{theorem}

\section{Two decompositions for functions\texorpdfstring{ in $\lp{2}(\N,\Phi)$}{}}
\label{sec_two_decomps}
In this section we establish several structural results about the space
\[
\lp{2}(\N,\Phi) \coloneqq \{ f \colon \N \to \C : \nbar f \nbar_\Phi < \infty \}
\]
where $\nbar \cdot \nbar_\Phi$ is the seminorm defined in \eqref{eqn:besSeminorm}.
In particular, we prove two ways in which elements of $\lp{2}(\N,\Phi)$ can be decomposed into pseudo-random and structured components.
These decomposition theorems will play crucial roles in the proof of \cref{thm_hilbertErdosNatural}.

Related decompositions of functions on $\N$ into orthogonal components have been studied in \cite{Host_Kra09} and \cite{Frantzikinakis15}.
However, those decompositions required some additional regularity on the function being decomposed and do not apply to all bounded functions on $\N$.
Also, similar but more quantitative decompositions are known for complex-valued functions over finite intervals $\{1,\ldots,N\}$ (cf.\ \cite{GT10-2}), but they don't possess qualitative (i.e.\ infinitary) analogues for functions over $\N$.

In \cref{sec_completeness_natural} we prove a completeness result for the space $\lp{2}(\N,\Phi)$.
Then in \cref{subsec_phiDichotomy} we introduce the space $\bes(\N,\Phi)$ of Besicovitch almost periodic functions along a \Folner{} sequence $\Phi$.
Members of $\bes(\N,\Phi)$ play the role of the structured part in our first decomposition result, \cref{thm_besicSplitting}.

Our second splitting, of functions from $\lp{2}(\N,\Phi)$ into compact and weak mixing functions, is based on the Jacobs--de Leeuw--Glicksberg splitting and is the topic of \cref{sec_jdlgSplitting_for_N}.

\subsection{A completeness lemma\texorpdfstring{ for $\lp{2}(\N,\Phi)$}{}}
\label{sec_completeness_natural}

Minkowski's inequality~\eqref{eqn:mink} implies that the space $\lp{2}(\N,\Phi)$ is a vector space over $\C$.
However $\lp{2}(\N,\Phi)$ is not a Hilbert space.
Indeed, $\nbar \cdot \nbar_\Phi$ is not a norm: the limit defining the inner product $\bilin{f}{h}_\Phi$ need not exist for all $f,h\in \lp{2}(\N,\Phi)$, and the space $\lp{2}(\N,\Phi)$ need not be complete with respect to $\nbar \cdot \|_\Phi$.
To address the latter issue, we make use of the following proposition.
We say that a sequence $j \mapsto f_j : \N \to \C$ of functions is \define{Cauchy} with respect to $\nbar \cdot \nbar_\Phi$ if, for all $\epsilon>0$, there exists $N\in \N$ such that for all $j,k \geq N$ one has $\nbar f_k-f_j\nbar_\Phi \leq \epsilon$.

\begin{proposition}
\label{prop_weak-completeness}
Let $j \mapsto f_j$ be a sequence in $\lp{2}(\N,\Phi)$ that is Cauchy with respect to $\nbar \cdot \nbar_\Phi$.
Then there exists a subsequence $\Psi$ of $\Phi$ and $f\in \lp{2}(\N,\Psi)$ such that $\nbar f - f_j \nbar_{\Psi} \to 0$ as $j \to \infty$.
Moreover, if all the $f_j$ take values in an interval $[a,b]$, then so does $f$.
\end{proposition}

\begin{remark}
If the \Folner{} sequence $\Phi$ satisfies $\Phi_N\subset \Phi_{N+1}$ for all $N\in\N$, then one can adapt the proof of \cite[II\,\S2]{MR0013443} to show that $\lp{2}(\N,\Phi)$ is complete with respect to $\nbar\cdot\nbar_\Phi$, meaning that any sequence of functions in $\lp{2}(\N,\Phi)$ that is Cauchy with respect to $\nbar\cdot\nbar_\Phi$ has a limit in $\lp{2}(\N,\Phi)$.
In particular, in this case it is not necessary to pass to a subsequence of $\Phi$.
We do not pursue this here for two reasons: on the one hand, the proof of \cref{prop_weak-completeness} is much shorter. On the other hand, we find it necessary to pass to subsequences of \Folner{} sequences frequently for many reasons, so we see no reason not to do so here as well.
\end{remark}

\begin{proof}[Proof of \cref{prop_weak-completeness}]
Since $j \mapsto f_j$ is Cauchy and all Besicovitch seminorms \eqref{eqn:besSeminorm} satisfy the triangle inequality, it suffices to find a subsequence $\Psi$ of $\Phi$ and a subsequence $j \mapsto f_{\sigma(j)}$ such that $\nbar f  - f_{\sigma(j)} \nbar_\Psi \to 0$ as $j \to \infty$.
To this end we assume, by passing to a subsequence if necessary, that for all $j\in\N$ and all $k\geq j$ we have $\nbar f_k-f_j\nbar_{\Phi}^2 \leq \tfrac{1}{j}$.
In particular, with $C \coloneqq(\nbar f_1\nbar_{\Phi}+1)^2$, the estimate $\nbar f_k\nbar_{\Phi}^2 \leq C$ is valid for all $k\in\N$.
Now, for every $k\in\N$, pick $N(k) \in\N$ such that $N(k+1) > N(k)$ for all $k \in \N$ and that, for all $N \geq N(k)$ and all $j\in\{1,\ldots,k\}$, one has
\[
\frac{1}{|\Phi_N|}\sum_{n\in\Phi_N} |f_j(n)-f_k(n)|^2\leq \frac{2}{j}\qquad\text{ and }\qquad\frac{1}{|\Phi_N|}\sum_{n\in\Phi_N} |f_j(n)|^2\leq 2C.
\]
Also, by further refining the subsequence $k \mapsto N(k)$ if necessary, we can assume that
\[
|\Phi_{N(k)}| > k^2 \max \left\{\sum_{n\in\Phi_{N(i)}}\big|f_k(n)-f_i(n)\big|^2:1\leq i<k\right\}
\]
for all $k > 1$.
Define the \Folner{} sequence $\Psi$ by $\Psi_k \coloneqq \Phi_{N(k)}$ for all $k \in \N$.

Let $\Xi_M \coloneqq \Psi_M \setminus \left(\bigcup_{k=1}^{M-1} \Psi_k\right)$ and set $\zeta_M \coloneqq \Psi_M \setminus \Xi_M$, the latter being a subset of $\bigcup_{i=1}^{M-1}\Psi_i$.
Define $f \colon \N \to \C$ by
\[
f(n)
\coloneqq
\sum_{M=1}^{\infty} 1_{\Xi_M}(n) f_M(n)
=
\begin{cases}
0, & \text{if}\ n \notin \bigcup\limits_{K=1}^\infty \Psi_K \\
f_M(n), & \text{if}\ M = \min \{ K \in \N : n \in \Psi_K \}
\end{cases}
\]
for all $n \in \N$.
By construction, $f$ takes values in an interval $[a,b]$ if all the functions $f_M$ do.
Using $|x+y|^2/2\leq|x|^2+|y|^2$, for each $j\leq M\in\N$ we have the estimate
\begin{align*}
\frac{1}{2} \sum_{n\in\Psi_M} |f_j(n)-f(n)|^2
&
\leq
\sum_{n\in\Psi_M} |f_j(n)-f_M(n)|^2+\sum_{n\in\zeta_M} |f_M(n)-f(n)|^2
\\
&
\leq
\frac{2|\Psi_M|}{j}+\sum_{i=1}^{M-1}\sum_{n\in\Xi_i}|f_M(n)-f_i(n)|^2
\\
&
\leq
\frac{2|\Psi_M|}{j}+\frac{|\Psi_M|}{M}
\end{align*}
which proves that $\nbar f-f_j\nbar_{\Psi}\leq 4/j$, which tends to $0$ as $j \to \infty$.
\end{proof}

We will also make use of the following version of Bessel's inequality.

\begin{lemma}[Bessel's inequality]\label{lemma_bessel}
Let $u_1,u_2,\dots$ be a sequence in $\lp{2}(\N,\Phi)$ such that $\nbar u_j \nbar_\Phi = 1$ for all $j \in \N$ and $\bilin{ u_j }{ u_k}_\Phi$ exists and is $0$ for all $j \ne k$.
If $u \in \lp{2}(\N,\Phi)$ is such that $\bilin{u}{u_j}_\Phi$ exists for all $j \in \N$, then
\[
\sum_{j=1}^\infty \big|\bilin{u}{u_j}_\Phi \big|^2
\leq
\nbar u \nbar_\Phi^2
\]
holds.
\end{lemma}

\begin{proof}
It suffices to show that
\begin{equation}
\label{eq_besselproof1}
\sum_{j=1}^J \big| \bilin{u}{u_j}_\Phi \big|^2
\leq
\nbar u \nbar_\Phi^2
\end{equation}
for every $J \in \N$.
Fix $N \in \N$ and write
\[
\bilinsq{f}{h}_N
=
\frac{1}{|\Phi_N|} \sum_{n \in \Phi_N} f(n) \, \overline{h(n)}
\]
for all $f,h \colon \N \to \C$.
Since $\bilinsq{f}{f}_N \ge 0$ for all $f \colon \N \to \C$ we have
\begin{align*}
0
\le
&
\bilinsq{
u - \sum_{j=1}^J u_j \bilinsq{u}{u_j}_N
}{
u - \sum_{k=1}^J u_k \bilinsq{u}{u_k}_N
}_N
\\
=
&
\bilinsq{u}{u}_N
-
2 \sum_{j =1}^J \big|\bilinsq{u}{u_j}_N\big|^2
+
\sum_{j,k=1}^J
\bilinsq{u}{u_j}_N \overline{\bilinsq{u}{u_k}_N} \bilinsq{u_j}{u_k}_N
\end{align*}
whence
\begin{equation}
\label{eq_proofbessel2}
2 \sum_{j=1}^J  \big|\bilinsq{u}{u_j}_N\big|^2
\le
\bilinsq{u}{u}_N
+
\sum_{j,k=1}^J
\bilinsq{u}{u_j}_N \overline{\bilinsq{u}{u_k}_N} \bilinsq{u_j}{u_k}_N
\end{equation}
holds.
Since the $u_j$ are pairwise orthogonal,
\[
\lim_{N \to \infty} \sum_{j,k=1}^J
\bilinsq{u}{u_j}_N \overline{\bilinsq{u}{u_k}_N} \bilinsq{u_j}{u_k}_N
-
\sum_{j=1}^J  \big|\bilinsq{u}{u_j}_N\big|^2
=
0.
\]
Taking the limit $N \to \infty$ in \eqref{eq_proofbessel2} gives \eqref{eq_besselproof1} as desired.
\end{proof}

\subsection{A general splitting technique\texorpdfstring{ for $\lp{2}(\N,\Phi)$}{}}
\label{subsec_phiDichotomy}

Our first decomposition result involves a notion of almost periodicity introduced over $\R$ by Besicovitch in \cite{Besicovitch26}.
We refer the reader to \cite{Besicovitch55, BL85} and the references therein for more on what have become known as Besicovitch almost periodic functions.
Over $\N$ they are defined as follows.

\begin{definition}
By a \define{trigonometric polynomial} we mean any function $a\colon \N \to \C$ of the form
\begin{equation}
\label{eqn:trigPoly}
a(n) = \sum_{j=1}^J c_j e^{2\pi i \theta_j n}
\end{equation}
for some $c_1,\dots,c_J \in \C$ and some \define{frequencies} $0 \le \theta_1,\dots,\theta_J < 1$.
A function $f \colon \N \to \C$ is \define{Besicovitch almost periodic} along $\Phi$ if, for every $\epsilon > 0$, one can find a trigonometric polynomial $a$ with $\nbar f - a \nbar_\Phi < \epsilon$.
\end{definition}

Write $\bes(\N,\Phi)$ for the set of all Besicovitch almost periodic functions along $\Phi$ and notice that $\bes(\N,\Phi)\subset \lp{2}(\N,\Phi)$.
The notion of pseudo-randomness complementary to Besicovitch almost periodicity is defined next.

\begin{definition}
The set $\bes(\N,\Phi)^\perp$ is defined to consist of those functions $f\in\lp{2}(\N,\Phi)$ such that
\[
\lim_{N \to \infty}  \frac{1}{|\Phi_N|} \sum_{n \in \Phi_N} f(n) e^{2\pi i n \theta} = 0
\]
for all frequencies $\theta\in[0,1)$.
\end{definition}

One can show directly from the definitions that $\bilin{f}{h}_\Phi = 0$ whenever $f\in\bes(\N,\Phi)$ and $h\in\bes(\N,\Phi)^\perp$.
Our main focus is the following splitting result.
Throughout this paper we will use $f_\anti$ to denote elements in $\bes(\N,\Phi)^\perp$.

\begin{theorem}
\label{thm_besicSplitting}
For every \Folner{} sequence $\Phi$ on $\N$ and any $f\in\lp{2}(\N,\Phi)$ there is a subsequence $\Psi$ of $\Phi$ and functions $f_\bes\in\bes(\N,\Psi)$ and $f_\anti  \in\bes(\N,\Psi)^\perp$ such that $f= f_\bes+f_\anti$.
Moreover, $f_\bes$ minimizes the distance between $f$ and $\bes(\N,\Psi)$ in the sense that
\[
\nbar f - f_\bes \nbar_\Psi = \inf\{\nbar f - g \nbar_\Psi: g\in \bes(\N,\Psi)\}
\]
and if $f$ takes values in an interval $[a,b]$, then so does $f_\bes$.
\end{theorem}
\begin{proof}
Combine \cref{thm_besicProjectionFamily} and \cref{thm_general_dichotomy} below.
\end{proof}

Instead of directly proving \cref{thm_besicSplitting}, we establish a general framework for decomposition results in $\lp{2}(\N,\Phi)$ that will in particular imply \cref{thm_besicSplitting}.
In fact, \cref{thm_besicSplitting} follows immediately from combining \cref{thm_besicProjectionFamily} and \cref{thm_general_dichotomy} below.

Suppose that for every \Folner{} sequence $\Phi$ we are given a $U(\Phi)$ of $\lp{2}(\N,\Phi)$ satisfying the following properties:
\begin{itemize}
\item
$U(\Phi)$ is a vector subspace of $\lp{2}(\N,\Phi)$;
\item
$U(\Phi)$ contains the constant functions and is closed under pointwise complex conjugation;
\item
for all $u,v\in U(\Phi)$ the inner product $\bilin{u}{v}_\Phi$ exists;
\item If $u,v\in U(\Phi)$ are real valued, then the function $n\mapsto\max\{u(n),v(n)\}$ is in $U(\Phi)$;
\item
$U(\Phi)$ is closed with respect to the topology on $\lp{2}(\N,\Phi)$ induced by the semi-norm $\nbar \cdot \nbar_\Phi$;
\item
if $\Psi$ eventually agrees with a subsequence of $\Phi$ then $U(\Psi)\supset U(\Phi)$.
\end{itemize}
Call any such assignment $U$ of subspaces to \Folner{} sequences a \define{projection family}.
Given a projection family one can consider, for each \Folner{} sequence $\Phi$, the subspace
\[
U(\Phi)^\perp \coloneqq \big\{ v\in\lp{2}(\N,\Phi): \bilin{u}{v}_\Phi \textup{ exists and equals $0$ for all } u\in U(\Phi) \big\}
\]
of $\lp{2}(\N,\Phi)$.
With a view towards proving \cref{thm_besicSplitting} we first verify that $\Phi \mapsto \bes(\N,\Phi)$ is a projection family.
The following fact can be viewed as von Neumann's ergodic theorem on the $1$-dimensional Hilbert space $\C$; we provide a short proof for the sake of completeness.
\begin{lemma}
\label{lemma_exponentialfolnervanish}
Let $\theta\in(0,1)$ and let $\Phi$ be a \Folner{} sequence. Then
\begin{equation}
\label{eqn_exponentialfolnervanish}
\lim_{N\to\infty}\frac1{|\Phi_N|}\sum_{n\in\Phi_N}e^{2\pi in\theta}=0.
\end{equation}
In particular $\bilin{a}{b}_\Phi$ exists for all trigonometric polynomials $a$ and $b$.
\end{lemma}
\begin{proof}
Let $N\in\N$ be large and let
\[
\epsilon_N \coloneqq \big|(\Phi_N+1)\triangle\Phi_N\big|/|\Phi_N|
\qquad
A_N \coloneqq \frac{1}{|\Phi_N|} \sum_{n\in\Phi_N}e^{2\pi in\theta}
\]
and
\[
B_N\coloneqq\frac1{|\Phi_N|}\sum_{n\in\Phi_N+1}e^{2\pi in\theta}.
\]
On the one hand $|A_N-B_N|\leq\epsilon_N$ but on the other hand $B_N=e^{2\pi i\theta}A_N$, which implies that $|A_N|<\epsilon_N/|1-e^{2\pi i\theta}|$. Since $\epsilon_N\to0$ we conclude that $A_N\to0$ as desired.

Now, if $a$ and $b$ are trigonometric polynomials then so is $n \mapsto a(n) \overline{b(n)}$ and the limit $\bilin{a}{b}_\Phi$ exists as it is a linear combination of constants and of limits of the form \eqref{eqn_exponentialfolnervanish}.
\end{proof}

\begin{theorem}
\label{thm_besicProjectionFamily}
The assignment $\Phi \mapsto \bes(\N,\Phi)$ is a projection family.
\end{theorem}
\begin{proof}
It follows from the triangle inequality that $\bes(\N,\Phi)$ is a subspace of $\lp{2}(\N,\Phi)$.
Since constant functions are trigonometric polynomials, and since the complex conjugation of a trigonometric polynomial remains such, it is immediate that $\bes(\N,\Phi)$ contains the constant functions and is closed under pointwise complex conjugation.

The fact that the space $\bes(\N,\Phi)$ is closed with respect to $\nbar\cdot\nbar_\Phi$ is an immediate consequence of the definition of $\bes(\N,\Phi)$ as the closure in $\lp{2}(\N,\Phi)$ of the space of trigonometric polynomials with respect to $\nbar\cdot\nbar_\Phi$.

Fix now $u,v$ in $\bes(\N,\Phi)$ both real-valued.
From the relation
\[
\max \{ u, v \} = \frac{1}{2}(u+v+|u-v|)
\]
and linearity, the fact that $\max \{ u, v \}$ belongs to $\bes(\N,\Phi)$ would follow from the knowledge that $|w|$ belongs to $\bes(\N,\Phi)$ whenever $w$ does.
That knowledge is the content of \cite[Lemma 5$^\circ$ in Chapter II, {\S}5]{Besicovitch55}; see also \cite{Bohr25-1,Bohr25-2}.
We give here a proof for completeness.
Fix $w \in \bes(\Phi,\N)$ and $\epsilon > 0$.
Let $a$ be a trigonometric polynomial with $\nbar u - a \nbar_\Phi < \epsilon/2$.
The reverse triangle inequality gives $\nbar |u| - |a| \nbar_\Phi < \epsilon/2$.
Apply the Stone-Weierstrass theorem to find a polynomial $b \in \C[z]$ with $|b(z) - |z|| < \epsilon/2$ for all $z \le \sup \{ |a(n)| : n \in \N \}$.
(This is possible because trigonometric polynomials have bounded range.)
The trigonometric polynomial $n \mapsto b(a(n))$ is then within $\epsilon$ of $|u|$ with respect to the $\nbar \cdot \nbar_\Phi$ semi-norm.

Next, we prove that $\bilin{u}{v}_\Phi$ exists for any $u,v\in\bes(\N,\Phi)$.
For this we use \cref{lemma_exponentialfolnervanish} and the inequality
\begin{align*}
&{}
\limsup_{N \to \infty} \left| \frac{1}{|\Phi_N|} \sum_{n \in \Phi_N} \Big( u(n) - w(n) \Big) \overline{v(n)} \right|
\\
\le{}
&{}
\nbar u - w \nbar_\Phi \sup \Big\{ \Big( \frac{1}{|\Phi_N|} \sum_{n \in \Phi_N} |v(n)|^2 \Big)^{1/2} : N \in \N \Big\}
\end{align*}
which is true for all $u,v,w \in \lp{2}(\N,\Phi)$ and implies continuity of $\bilin{\cdot}{\cdot}_\Phi$ in the first variable.
Fix $u \in \bes(\Phi,\N)$ and a trigonometric polynomial $a$.
Fix a sequence $n \mapsto b_n$ of trigonometric polynomials converging to $u$ with respect to $\nbar \cdot \nbar_\Phi$.
The sequence $n \mapsto b_n$ is Cauchy for $\nbar \cdot \nbar_\Phi$ so $n \mapsto \bilin{b_n}{a}_\Phi$ is Cauchy by the Cauchy-Schwarz inequality.
Denote by $\alpha$ its limit.
The above inequality implies that $\bilin{u}{a}_\Phi = \alpha$.

A similar inequality gives continuity of the form $\bilin{\cdot}{\cdot}_\Phi$ in the second variable, and the above argument can be repeated to prove that if $u,v \in \bes(\N,\Phi)$ and $c_n$ are trigonometric polynomials converging to $v$ with respect to $\nbar \cdot \nbar_\Phi$ then $\bilin{u}{v}_\Phi$ is the limit of the Cauchy sequence $n \mapsto \bilin{u}{c_n}_\Phi$.

Lastly, since $\nbar f \nbar_\Psi \le \nbar f \nbar_\Phi$ for all $f : \N \to \C$ whenever $\Psi$ eventually agrees with a subsequence of $\Phi$, it is immediate that $\bes(\Psi,\N) \supset \bes(\Phi,\N)$ whenever $\Psi$ eventually agrees with a subsequence of $\Phi$.
\end{proof}

In view of \cref{thm_besicProjectionFamily}, the following general decomposition result extends \cref{thm_besicSplitting}.

\begin{theorem}
\label{thm_general_dichotomy}
Let $U$ be a projection family and let $\Phi$ be a \Folner{} sequence.
For every $f \in \lp{2}(\N,\Phi)$ there exists a subsequence $\Psi$ of $\Phi$ and there is $f_U\in U(\Psi)$ such that:
\begin{enumerate}
  \item $f - f_U\in U(\Psi)^\perp$,
  \item $f_U$ minimizes the distance between $f$ and $U(\Psi)$ in the sense that $\nbar f - f_U \nbar_\Psi = \inf \{\nbar f - g \nbar_\Psi: g\in U(\Psi)\}$,
  \item if $f$ takes values in an interval $[a,b]$ then $f_U$ takes values in $[a,b]$.
\end{enumerate}
\end{theorem}

\cref{thm_general_dichotomy} would be immediate if $\lp{2}(\N,\Phi)$ were a Hilbert space and $U(\Phi)$ were a closed subspace, because then one could simply define $f_U$ as the orthogonal projection of $f$ onto $U(\Phi)$.
However, $\lp{2}(\N,\Phi)$ is not a Hilbert space, which requires us to overcome some difficulties.
In particular, it is problematic that $\bilin{f}{u}_\Phi$ may not exist for all $u \in U(\Phi)$.
The following technical lemma offers a way around this issue.

\begin{lemma}
\label{lem_spectrumn-is-coutable}
Let $U$ be a projection family and let $\Phi$ be a \Folner{} sequence.
For every $f\in\lp{2}(\N,\Phi)$ there exists a subsequence $\Psi$ of $\Phi$ such that the inner product $\bilin{f}{u}_{\Psi}$ exists whenever $u\in U(\Psi)$.
\end{lemma}
\begin{proof}
Fix $f\in\lp{2}(\N,\Phi)$.
We start with an inductive construction.
Put $u_0\coloneqq 0$ and $\Phi^{(0)}\coloneqq \Phi$.
Certainly $u_0 \in U(\Phi^{(0)})$ and $\bilin{f}{u_0}_{\Phi^{(0)}}$ exists.
Suppose for some $k \in \N$ that we have defined functions $u_0,\dots,u_{k-1} \in U(\Phi^{(k-1)})$ and \Folner{} sequences $\Phi^{(0)},\dots,\Phi^{(k-1)}$, each a subsequence of the previous one, such that $\bilin{f}{u_i}_{\Phi^{(k-1)}}$ exists for all $0 \le i \le k-1$.
For each \Folner{} subsequence $\Phi'$ of $\Phi^{(k-1)}$, let
\[
O_{k-1}(\Phi')\coloneqq \big\{u\in U(\Phi'): \bilin{u}{u_i}_{\Phi'}=0,~\forall i\in \{0,\ldots,k-1\}\big\}
\]
which is a linear subspace of $U(\Phi')$ that contains the constant functions.

We now distinguish two cases depending on whether or not there are a subsequence $\Phi'$ of $\Phi^{(k-1)}$ and a member $u$ of $O_{k-1}(\Phi')$ with $\nbar u \nbar_{\Phi'} \ne 0$.

In the first case we assume, for every subsequence $\Phi'$ of $\Phi^{(k-1)}$, that every $u \in U(\Phi')$ satisfying $\bilin{u}{u_i}_{\Phi'}$ for all $0 \le i \le k-1$ has the property $\nbar u \nbar_{\Phi'} = 0$.
If this happens we terminate our inductive construction, the result being a \Folner{} sequence $\Phi^{(k-1)}$ and a collection $u_0,\dots,u_{k-1}$ of members of $U(\Phi^{(k-1)})$ such that $\bilin{f}{u_i}_{\Phi^{(k-1)}}$ exists for all $0 \le i \le k-1$.

We claim in this first case that the conclusion of the lemma is true with $\Psi = \Phi^{(k-1)}$.
Fix $u \in U(\Psi)$.
The function
\[
v = u - \sum_{i=0}^{k-1} u_i \bilin{u}{u_i}_\Psi
\]
belongs to $O_{k-1}(\Psi)$ and therefore has a $\nbar \cdot \nbar_\Psi$ norm of zero.
It follows that
\begin{align*}
&
\frac{1}{|\Psi_N|} \sum_{n \in \Psi_N} f(n) \overline{u(n)}
\\
={}
&
\frac{1}{|\Psi_N|} \sum_{n \in \Psi_N} f(n) \overline{v(n)}
+
\sum_{i=0}^{k-1} \frac{1}{|\Psi_N|} \sum_{n \in \Psi_N} f(n) \overline{u_i(n)} \bilin{u}{u_i}_\Psi
\end{align*}
converges as $N \to \infty$ as desired.

In the second case we assume there is a subsequence $\Phi'$ of $\Phi^{(k-1)}$ and a member $u$ of $O_{k-1}(\Phi')$ with $\nbar u \nbar_{\Phi'} \ne 0$.
If this happens then the set
\begin{equation*}
Q_k
\coloneqq
\left\{
\left| \bilin{f}{u}_{\Phi'} \right|
:
\begin{aligned}
& \Phi' \text{ is a \Folner{} subsequence of } \Phi^{(k-1)} \\
& u \in O_{k-1}(\Phi') \text{  with } \nbar u\nbar_{\Phi'}=1 \\
& \bilin{f}{u}_{\Phi'} \text{ exists}
\end{aligned}
\right\}
\end{equation*}
is non-empty.
Indeed if, for some subsequence $\Phi'$ of $\Phi^{(k-1)}$, one can find some member $u$ of $O_{k-1}(\Phi')$ with $\nbar u \nbar_{\Phi'} \ne 0$, then note that $u/\nbar u \nbar$ belongs to $O_{k-1}(\Xi)$ for every subsequence $\Xi$ of $\Phi'$ and that $\bilin{f}{u}_\Xi$ will exist for a suitable choice of $\Xi$.

Write $\delta_k$ for the supremum of $Q_k$, which will be at most $\nbar f \nbar_\Phi$ by Cauchy-Schwarz.
Choose a \Folner{} subsequence $\Phi^{(k)}$ of $\Phi^{(k-1)}$ and $u_k\in O_{k-1}(\Phi^{(k)})$ with $\nbar u_k\nbar_{\Phi^{(k)}} = 1$ such that $\bilin{f}{u_k}_{\Phi^{(k)}}$ exists and $|\bilin{f}{u_k}_{\Phi^{(k)}}|>\delta_k-\frac{1}{k}$.
Then $\bilin{f}{u_i}_{\Phi^{(k)}}$ exists for all $0 \le i \le k$.

This concludes the consideration of the second case, and the inductive construction.
If, at any stage, we find ourselves in the first case discussed above then the proof is complete.
We therefore find ourselves with a sequence $u_0,u_1,\dots$ of functions, a sequence $\Phi^{(0)},\Phi^{(1)},\dots$ of \Folner{} sequences, and a sequence $\delta_1,\delta_2,\dots$ of suprema, as described in the second case.

Define $\Psi_N\coloneqq \Phi_N^{(N)}$.
The sequence $\Psi$ is a subsequence of $\Phi^{(1)}$ and is therefore itself a \Folner{} sequence.
We claim that for every $u\in U(\Psi)$ the inner product $\bilin{f}{u}_{\Psi}$ exists.
More precisely, we claim that
$$
\bilin{f}{u}_{\Psi}=\sum_{i=1}^\infty \bilin{f}{u_i}_{\Psi}\overline{\bilin{u}{u_i}_{\Psi}}.
$$
Note that the terms in the above series are well defined, since $\bilin{u}{u_i}_{\Psi}$ exists because $u,u_i\in U(\Psi)$ and $\bilin{f}{u_i}_{\Psi}$ exists by construction of $\Psi$.
Moreover, this series is absolutely convergent, because \cref{lemma_bessel} implies that the sequences $i \mapsto \bilin{f}{u_i}_{\Psi}$ and $i \mapsto \bilin{u}{u_i}_\Psi$ are in $\ell^2(\N)$.

For each $k\in\N$, define
\[
v_k\coloneqq u - \sum_{i=1}^{k-1}u_i \bilin{u}{u_i}_{\Psi}
\]
and observe that $v_k\in O_{k-1}(\Psi)$ and that $\nbar v_k\nbar_{\Psi}\leq \nbar u\nbar_{\Psi}$.
Therefore
\begin{eqnarray*}
  &&\limsup_{N\to\infty}\left|\frac1{|\Psi_N|}\sum_{n\in\Psi_N} f(n)\overline{u(n)}-\sum_{i=1}^\infty \bilin{f}{u_i}_{\Psi}\overline{\bilin{u}{u_i}_{\Psi}}\right|
  \\&\leq&
  \limsup_{N\to\infty}\left|\frac1{|\Psi_N|}\sum_{n\in\Psi_N} f(n)\overline{v_k(n)}\right|+\left|\sum_{i=k}^\infty \bilin{f}{u_i}_{\Psi}\overline{\bilin{u}{u_i}_{\Psi}}\right|
  \\&\leq&
  \delta_k\nbar v_k\nbar_{\Psi}+\left|\sum_{i=k}^\infty \bilin{f}{u_i}_{\Psi}\overline{\bilin{u}{u_i}_{\Psi}}\right|.
\end{eqnarray*}

It thus suffices to show that $\delta_k \to 0$ as $k\to\infty$.
But by \cref{lemma_bessel}, we get
$$
\nbar f\nbar_{\Psi}^2\geq \sum_{k=1}^\infty |\bilin{f}{u_k}_{\Psi}|^2\geq \sum_{k=1}^\infty \big(\delta_k -\tfrac{1}{k}\big)^2
$$
and since $f\in\lp{2}(\N,\Phi)$, the series converges, which implies that indeed $\delta_k \to 0$ as $k\to\infty$.
\end{proof}

\begin{proof}[Proof of \cref{thm_general_dichotomy}]
As guaranteed by \cref{lem_spectrumn-is-coutable}, let $\Psi$ be a \Folner{} subsequence of $\Phi$ such that for every $u\in U(\Psi)$ the limit $\bilin{f}{u}_\Psi$ exists.
Define
\[
\delta\coloneqq \inf\big\{\nbar f-u\nbar_{\Psi}^2: u\in U(\Psi)\big\}.
\]
For each $k\in\N$ choose $u_k\in U(\Psi)$ with
$\nbar f-u_k\nbar_{\Psi}^2 < \delta+\frac{1}{k}$.

If $f$ takes values in $[a,b]$, then we can replace $u_k$ with the function
$$
v_k:n\mapsto\begin{cases}
a, &\text{if}~ \Re u_k(n) < a,
\\
\Re u_k(n), &\text{if}~ a\leq \Re u_k(n) \leq b,
\\
b, &\text{if}~ \Re u_k(n) > b,
\\
\end{cases}
$$
where $\Re z$ denotes the real part of a complex number $z$.
Indeed, it is clear that $\nbar f-v_k\nbar_{\Psi}^2\leq \nbar f-u_k\nbar_{\Psi}^2 < \delta+\frac{1}{k}$.
On the other hand, it follows from the definition of a projection family that $U(\Psi)$ is closed under the operation of the pointwise minimum, so $v_k$ still belongs to $U(\Phi)$.
Therefore we can assume without loss of generality that when $f$ takes values in $[a,b]$, then so do the functions $u_k$.

Next, an application of the parallelogram law to the vectors $f-u_j$ and $f-u_k$ shows that $\nbar u_j-u_k\nbar_{\Psi}^2 \leq\frac2{j}+\frac2{k}$, which implies that $(u_k)_{k\in\N}$ is a Cauchy sequence with respect to $\nbar \cdot \nbar_{\Psi}$.
Using \cref{prop_weak-completeness} and by refining $\Psi$ if necessary, we can find $f_U\in\lp{2}(\N,\Psi)$ such that $\lim_{k\to\infty}\nbar f_{U}-u_k\nbar_{\Psi}=0$.
If $f$ takes values in $[a,b]$ (and hence so do all the $u_k$), then $f_U$ also takes values in $[a,b]$.
Since $U(\Psi)$ is closed, it follows that $f_U$ belongs to $U(\Psi)$.
Minkowski's inequality implies that $\nbar f-f_U \nbar_{\Psi}^2=\delta$.
In particular, $f_U$ minimizes the distance between $f$ and $U(\Psi)$.

Write $h \coloneqq f-f_U$.
We claim that $h$ belongs to $U(\Psi)^\perp$.
First note that $\bilin{h}{u}_\Psi$ exists for all $u \in U(\Psi)$ because both $\bilin{f}{u}_\Psi$ and $\bilin{f_U}{u}_\Psi$ exist.
Next, fix $u\in U(\Psi)$ with $\|u\|_\Psi\leq 1$ and define $I \coloneqq \bilin{h}{u}_\Psi$.
We have
\begin{align*}
&
\big\nbar h  - I u\big\nbar_{\Psi}^2
\\
=
&
\lim_{N\to\infty}\frac{1}{|\Psi_{N}|} \sum_{n\in\Psi_{N}} |h(n)|^2-  h(n) \overline{I u(n)}-  \overline{h(n)}I u(n)+|I|^2 |u(n)|^2
\\
\leq
&
\nbar h \nbar_{\Psi}^2- |I|^2 (2-\nbar u\nbar_{\Psi}^2).
\end{align*}
Since $\nbar u\nbar_{\Psi}^2\leq 1$ and $\nbar h \nbar_{\Psi}^2=\delta$, we conclude that $\nbar h \nbar_{\Psi}^2 - |I|^2 (2-\nbar u\nbar_{\Psi}^2)\leq \delta -|I|^2$.
Therefore
\begin{equation}
\label{eq_Bohr-aperiodic-splitting-1}
\big\nbar h - I u\big\nbar_{\Psi}^2\leq \delta  -|I|^2.
\end{equation}
On the other hand, $h - Iu = f - (f_U + I u)$ and $f_U + I u \in U(\Psi)$.
So
\begin{equation}
\label{eq_Bohr-aperiodic-splitting-2}
\big\nbar h - I u\big\nbar_{\Psi}^2\geq \delta.
\end{equation}
Combining \eqref{eq_Bohr-aperiodic-splitting-1} and \eqref{eq_Bohr-aperiodic-splitting-2} proves that $I=0$.
\end{proof}

\begin{remark}
\label{rem_values_in_interval}
Under the assumptions of \cref{thm_general_dichotomy}, the function $f_U\in U(\Psi)$ is unique in the following two senses:
\begin{enumerate}	
[label=(\alph{enumi}),ref=(\alph{enumi}),leftmargin=1cm]
\item
\label{itm_vii1}
If $f_{U}'\in U(\Psi)$ is such that $f-f_U'\in U(\Psi)^\perp$ then $\nbar f_U-f_U'\nbar_\Psi=0$.
\item
\label{itm_vii2}
If $f_U'\in U(\Psi)$ also minimizes the distance between $f$ and $U(\Psi)$ (i.e.\ $\nbar f - f_U' \nbar_\Psi = \inf \{\nbar f - g \nbar_\Psi: g\in U(\Psi)\}$), then $\nbar f_U-f_U'\nbar_\Psi=0$.
\end{enumerate}
In the second half of the proof of \cref{thm_general_dichotomy} we show that a function $f_U'\in U(\Psi)$ that minimizes the distance between $f$ and $U(\Psi)$ must satisfy $f-f_U'\in U(\Psi)^\perp$; therefore
part \ref{itm_vii2} follows from part \ref{itm_vii1}.

To verify part \ref{itm_vii1}, note that $f-f_U,f-f_U'\in U(\Psi)^\perp$ implies that $f_U-f_U'\in U(\Psi)^\perp$, while $f_U,f_U'\in U(\Psi)$ implies that $f_U-f_U'$, and therefore $\|f_U-f_U'\|^2={\bilin{f_U-f_U'}{f_U-f_U'}}=0$.
\end{remark}

We conclude this subsection with a small detour on the further applicability of \cref{thm_general_dichotomy}; this remarks are unrelated to the proof of \cref{thm_erdosNatural}.

By a \define{nilsystem} we mean a pair $(G/\Gamma,g)$ where $G$ is a nilpotent Lie group, $\Gamma$ is a discrete, co-compact subgroup of $G$, and $g \in G$ acts on $G/\Gamma$ by left multiplication.
A function $\alpha \colon \N \to \C$ is a \define{basic nilsequence} if there exists a nilsystem $(G/\Gamma,g)$ and a continuous function $F\colon G/\Gamma\to\C$ such that $\alpha(n)=F(g^n\Gamma)$.
Call a function $f\in\lp{2}(\N,\Phi)$ a \define{Besicovitch nilsequence} along $\Phi$ if for every $\epsilon>0$ there exists a basic nilsequence $\alpha\colon \N\to\C$ such that $\|f-\alpha\|_\Phi<\epsilon$.

Denote by $U(\Phi)$ the family of all Besicovitch nilsequences with respect to $\Phi$.
Since the \Cesaro{} average of a basic nilsequence along any \Folner{} sequence exists (cf.\ \cite{MR2122919}) one can easily adapt the proof of \cref{thm_besicProjectionFamily} to show that the assignment $\Phi\mapsto U(\Phi)$ is a projection family.

A function $f \colon \N \to \C$ is a \define{good weight} for the polynomial multiple ergodic theorem if, for every probability space $(X,\mathcal{B},\mu)$ and any commuting, measure-preserving transformations $T_1,\dots,T_k \colon X \to X$ the quantity
\[
\lim_{N \to \infty} \frac{1}{|\Psi_N|} \sum_{n \in \Psi_N}f_{\mathsf{w}}(n) \, T_1^{p_1(n)}h_1\cdots T_k^{p_k(n)}h_k\d\mu
\]
exists and equals
\[
\lim_{N \to \infty} \frac{1}{|\Psi_N|} \sum_{n \in \Psi_N}T_1^{p_1(n)} h_1 \cdots T_k^{p_k(n)}h_k \d\mu
\]
for any polynomials $p_1,\dots,p_k \in \Z[x]$ and any $h_1,\dots,h_k \in \lp{\infty}(X,\mathcal{B},\mu)$.

Combining the fact that $U(\Phi)$ is a projection family with \cref{thm_general_dichotomy} and \cite[Theorem 1.2]{Frantzikinakis15} we deduce the following result.
\begin{theorem}
Let $\Phi$ be a \Folner{} sequence on $\N$ and let $f\in\lp{2}(\N,\Phi)$.
Then there exists a subsequence $\Psi$ of $\Phi$ and a decomposition $f=f_{\mathsf{nil}}+f_{\mathsf{w}}$ such that $f_{\mathsf{nil}}$ is a Besicovitch nilsequence with respect to $\Psi$ and $f_{\mathsf{w}}$ is a good weight for the polynomial multiple ergodic theorem.
\end{theorem}

\subsection{A version of the Jacobs--de Leeuw--Glicksberg splitting\texorpdfstring{ for $\lp{2}(\N,\Phi)$}{}}
\label{sec_jdlgSplitting_for_N}

The second decomposition theorem that we use in the proof of \cref{thm_hilbertErdosNatural}, which represents $1_A$ as a sum of a weak mixing function and a compact function, can be viewed as a discrete version of the Jacobs--de Leeuw--Glicksberg splitting on Hilbert spaces.
After recalling this splitting and introducing versions of weak mixing and compactness for functions in $\lp{2}(\N,\Phi)$ we prove the main result of this section, \cref{thm_jdlgSplitting}.

Fix an isometry $\umult$ on a Hilbert space $(\Hilb, \nbar \cdot \nbar_\Hilb)$.

\begin{definition}
\label{def_compact_on_Hilb}
An element $x \in \Hilb$ is \define{compact} if $\{ \umult^n x: n \in \N \}$ is a pre-compact subset of $(\Hilb, \nbar \cdot \nbar_\Hilb)$.
Equivalently, $x$ is compact if for all $\epsilon>0$ there exists $K \in \N$ such that
\[
\min \{ \nbar \umult^m x - \umult^k x \nbar_\Hilb : 1 \le k \le K \} \leq \epsilon
\]
for all $m \in \N$.
\end{definition}

\begin{definition}
\label{def_weakly_mixing_on_Hilb}
An element $x \in \Hilb$ is called \define{weak mixing} if for all $\epsilon> 0$ and all $y \in \Hilb$ the set $\left\{ n \in \N : |\bilin{\umult^n x}{y}| \geq \epsilon \right\}$ has zero density with respect to every \Folner{} sequence on $\N$.
\end{definition}

The set of all compact elements in $\Hilb$, denoted $\Hilb_\comp$, is a closed and $\umult$ invariant subspace of $\Hilb$, as is the set $\Hilb_{\wm}$ of weak mixing elements.
The principle that $\Hilb$ splits into the direct sum of $\Hilb_{\compact}$ and $\Hilb_{\wm}$ traces back as far as the works of Koopman and von Neumann \cite{Koopman_vonNeumann32} (see also \cite[Theorem 2.3]{Bergelson96}) and was later pushed to greater generality by work of Jacobs~\cite{MR0077092} and de~Leeuw, Glicksberg~\cite{MR0131784} (see also \cite[Chapter 2.4]{MR0797411} and \cite[Example~16.25]{MR3410920}).

\begin{theorem}[The Jacobs-de Leeuw-Glicksberg splitting]
\label{thm_JdLG_for_N}
Let $\umult$ be an isometry on a Hilbert space $\Hilb$.
Then $\Hilb_\comp$ and $\Hilb_\wm$ are orthogonal spaces and $\Hilb=\Hilb_{\compact}\oplus \Hilb_{\wm}$.
In particular, for any $x\in\Hilb$ there exist $x_\compact\in\Hilb_\compact$ and $x_\wm\in\Hilb_\wm$ such that $x=x_\comp + x_\wm$.
\end{theorem}

Let us introduce now the analogous notions of compact and weak mixing for elements in $\lp{2}(\N,\Phi)$.
Recall that, given $f \colon \N \to \C$, we write $\rmult^m f$ for the function $n\mapsto f(m+n)$.
One should think of $\rmult^1$ acting on $\lp{2}(\N,\Phi)$ as playing the role of the isometry $\umult$ on $\Hilb$ in \cref{thm_JdLG_for_N}.

\begin{definition}
A function $f \in\lp{2}(\N,\Phi)$ is \define{compact} along $\Phi$ if, for every $\epsilon > 0$, one can find $K \in \N$ such that
\[
\min \{ \nbar \rmult^m f - \rmult^k f \nbar_\Phi : 1 \le k \le K \} < \epsilon
\]
for all $m \in \N$.
\end{definition}

Observe that any trigonometric polynomial is compact along any $\Phi$.
Since compact functions form a closed subset of $\lp{2}(\N,\Phi)$, every $f\in\bes(\N,\Phi)$ is compact along $\Phi$.
We remark that one can show the set of functions compact along $\Phi$ is in fact a subspace of $\lp{2}(\N,\Phi)$.

\begin{definition}
A function $f \in\lp{2}(\N,\Phi)$ is \define{weak mixing} along $\Phi$ if, for every bounded function $h \colon \N \to \C$ and every subsequence $\Psi$ of $\Phi$ such that $\bilin{\rmult^n f}{h}_\Psi$ exists for all $n \in \N$, one has
\[
\upperdens_\Psi\Big(\Big\{ n \in \N : \big|\bilin{\rmult^n f}{h}_\Psi \big|>\epsilon \Big\}\Big) = 0
\]
for all $\epsilon > 0$.
\end{definition}

\begin{lemma}
\label{lem:waltersDensity}
If $f \in \lp{2}(\N,\Phi)$ is weak mixing along $\Phi$ then
\[
\lim_{N \to \infty} \frac{1}{|\Psi_N|} \sum_{n \in \Psi_N} \left| \bilin{\rmult^n f}{h}_\Psi \right| = 0
\]
for all subsequences $\Psi$ of $\Phi$ and all $h \in \lp{2}(\N,\Psi)$ such that $\bilin{\rmult^n f}{h}_\Psi$ exists for all $n \in \N$.
\end{lemma}
\begin{proof}
Fix $f \in \lp{2}(\N,\C)$ that is weak mixing along $\Phi$.
Fix also a subsequence $\Psi$ of $\Phi$ and $h \in \lp{2}(\N,\Psi)$ such that $\bilin{\rmult^n f}{h}_\Psi$ exists for all $n \in \N$.
The sequence $a(n) = \bilin{\rmult^n f}{h}_\Psi$ is bounded.
The implication $(ii) \Rightarrow (i)$ of \cite[Theorem~1.20]{MR648108} and its proof are valid for averages along any \Folner{} sequence.
But $(ii)$ therein follows from our hypothesis on $f$.
\end{proof}

\begin{lemma}
Let $\Phi$ be a \Folner{} sequence and let $f,h\in\lp{2}(\N,\Phi)$ be compact and weak mixing along $\Phi$, respectively.
Then $\bilin fh_\Phi=0$.
\end{lemma}
\begin{proof}
If $\nbar f \nbar_\Phi = 0$ or $\nbar h \nbar_\Phi = 0$ then the result follows from Cauchy-Schwarz.
Otherwise, choose a subsequence $\Psi$ of $\Phi$ such that
$\bilin fh_\Psi$ exists.
Passing to a further subsequence if needed, we will also assume that all the inner products $\bilin{\rmult^nf}{\rmult^m h}_\Psi$ exist.
After scaling if needed, we will further assume that $\nbar f\nbar_\Psi=\nbar h\nbar_\Psi=1$.

Fix $\epsilon>0$ and choose $K$ so that for every $m\in\N$, there is some $1 \le k\leq K$ with $\|\rmult^mf-\rmult^kf\|_\Phi<\epsilon$.
Therefore
\[
\big|\bilin{f}{h}_\Psi\big|= \big|\bilin{\rmult^mf}{\rmult^mh}_\Psi\big| \leq\epsilon+\big|\bilin{\rmult^kf}{\rmult^mh}_\Psi\big| \leq\epsilon+\sum_{k=1}^K\big|\bilin{\rmult^kf}{\rmult^mh}_\Psi\big|
\]
holds.
Since $h$ is weak mixing, we conclude that
\[
\big|\bilin{f}{h}_\Psi\big|\leq\epsilon+\sum_{k=1}^K \limsup_{N\to\infty}\frac1{|\Psi_N|}\sum_{m\in\Psi_N} \big|\bilin{\rmult^kf}{\rmult^mh}_\Psi\big|=\epsilon
\]
via \cref{lem:waltersDensity}.
Since $\epsilon$ was arbitrary, we obtain $\bilin fh_\Psi=0$.
Since we chose $\Psi$ as an arbitrary subsequence of $\Phi$ for which all $\bilin{\rmult^n f}{\rmult^m h}_\Psi$ exist, it follows that $\bilin fh_\Phi=0$.
\end{proof}

Any Besicovitch almost periodic function is compact and therefore, if $h$ is weak mixing along $\Phi$, then $\bilin{h}{f}_\Phi=0$ for all $f\in\bes(\N,\Phi)$ and hence $h\in\bes(\N,\Phi)^\perp$.

\begin{remark}
The condition of a function $f$ being weak mixing is similar to (but slightly weaker than) the condition that the Host--Kra local seminorm $\|f\|_{\Phi,2}$  of $f$ equals $0$ in the sense of \cite[Definition 2.3]{Host_Kra09}.
We stress that this is weaker than the uniformity seminorm $\|f\|_{U(2)}$ of $f$ equaling $0$ in the sense of \cite[Definition 2.6]{Host_Kra09}.
In fact, \cite[Corollary 2.18]{Host_Kra09} implies that $\|f\|_{U(2)}=0$ is equivalent to $f\in\bes(\N,\Phi)^\perp$ for every \Folner{} sequence $\Phi$.
\end{remark}

As the following example shows (see also the example in \cite[Section 2.4.3]{Host_Kra09}) there are functions in $\bes(\N,\Phi)^\perp$ that are compact.

\begin{example}
\label{example_1}
We will now construct a bounded function $f\colon\N\to\C$ and a \Folner{} sequence $\Phi$ such that $f$ is simultaneously compact along $\Phi$ and a member of $\bes(\N,\Phi)^\perp$.
Let $k \mapsto N_k$ be an increasing sequence of natural numbers with $N_{k-1}/N_k \to 0$ as $k\to\infty$.
 Assume $f$ has already been defined on the interval $[1,N_k)$. Then we define $f$ on the interval $[N_k,N_{k+1})$ by
\[
f(n)\coloneqq
\begin{cases}
(-1)^n,&\text{if }n\in \left[ N_k,\big\lfloor\tfrac{N_{k+1}}{2}\big\rfloor\right)
\\
\\
-(-1)^{n},&\text{if }n\in \left[\big\lfloor\tfrac{N_{k+1}}{2} \big\rfloor,N_{k+1}\right)
\end{cases}
\]
for all $N_k \le n < N_{k+1}$.
Also, let $\Phi$ denote the \Folner{} sequence given by $\Phi_k\coloneqq [1,N_k]$ for all $k\in\N$.
It is then easy to verify that $\nbar T^2 f-f\nbar_{\Phi}=0$ and hence $f$ is compact with respect to $\Phi$.
However, using \cref{lemma_exponentialfolnervanish} when $\theta \ne \frac{1}{2}$ and direct calculation when $\theta = \frac{1}{2}$, one can show that $\bilin{f}{e_\theta}_\Phi=0$ for all $\theta\in\T$, where $e_\theta(n)\coloneqq e^{2\pi i n\theta}$, which implies that $f\in\bes(\N,\Phi)^\perp$.
\end{example}

Our second splitting theorem is as follows.

\begin{theorem}
\label{thm_jdlgSplitting}
For every $f\in\lp{2}(\N,\Phi)$ there is a subsequence $\Psi$ of $\Phi$ and functions $f_\comp,f_\wm \in\lp{2}(\N,\Psi)$ with $f_\comp$ compact along $\Psi$, $f_\wm$ weak mixing along $\Psi$, and $f = f_\comp + f_\wm$.
Moreover, if $f$ is real-valued with $a\leq f \leq b$ for some $a\leq b$ then $f_\comp$ is real-valued and satisfies $a\leq f_\comp \leq b$.
\end{theorem}

\begin{remark}
The conclusion of \cref{thm_jdlgSplitting} is similar to that of \cref{thm_general_dichotomy}.
We remark that, in fact, $f_\comp$ minimizes the distance between $f$ and the closed subspace of compact functions in $\lp{2}(\N,\Phi)$ but will not make use of this.
It is also true that $f_\comp$ can be shown to be unique in the sense of parts (a) and (b) of \cref{rem_values_in_interval}.
\end{remark}

The proof of \cref{thm_jdlgSplitting} requires some lemmas, the first of which is essentially \cite[Lemma~4.23]{MR603625}.
Recall that a triple $(X,\mu,T)$ is a \define{measure preserving system} if $X$ is a compact space equipped with a Borel probability measure $\mu$ and $T : X \to X$ is a measurable map that preserves $\mu$.
Given a measure preserving system $(X,\mu,T)$ one can consider the Hilbert space $\lp{2}(X,\mu)$ whose norm is denoted $\nbar \cdot \nbar_\mu$.
The map $T$ induces an isometry $\umult$ on $\lp{2}(X,\mu)$ defined by $Uf = f \circ T$ for all $f \in \lp{2}(X,\mu)$.

\begin{lemma}
\label{lem_compact_functions}
Let $(X,\mu,T)$ be a measure preserving system.
For the isometry $\umult f = f \circ T$ of the Hilbert space $\lp{2}(X,\mu)$ the constant functions are compact, $|\phi|$ is compact whenever $\phi$ is, and both $\min \{ \phi,\psi \}$ and $\max \{ \phi,\psi \}$ are compact whenever $\phi,\psi$ are compact and real-valued.
\end{lemma}
\begin{proof}
Since the constant functions are fixed points of $\umult$ they certainly satisfy \cref{def_compact_on_Hilb}.
The reverse triangle inequality gives
\begin{align*}
\nbar \umult^m(|\phi|) - \umult^k(|\phi|) \nbar_\mu^2
&
=
\int_X \Big| |\phi(\tmult^m x)| - |\phi(\tmult^k x)| \Big|^2 \d \mu(x)
\\
&
\le
\int_X \Big| \phi(\tmult^m x) - \phi(\tmult^k x) \Big|^2 \d \mu(x)
=
\nbar \umult^m(\phi) - \umult^k(\phi) \nbar_\mu^2
\end{align*}
so compactness of $\phi$ implies compactness of $|\phi|$.
For the last claim write
\[
\min \{\phi,\psi\} = \frac{\phi + \psi-|\phi - \psi|}{2}\quad \text{and}\quad \max \{\phi,\psi\} = \frac{\phi + \psi+|\phi - \psi|}{2}
\]
pointwise.
\end{proof}

\begin{corollary}
\label{cor_compact_part_positive_bounded}
Under the hypothesis of \cref{lem_compact_functions} if $a \le \phi \le b$ for some $a\leq b $ then $a \le \phi_\comp \le b$.
\end{corollary}
\begin{proof}
Since $\phi_\comp$ is the orthogonal projection of $\phi$ on $\Hilb_\comp$ it is characterized as the unique element of $\Hilb_\comp$ closest to $\phi$.
Since the real part of $\phi_\comp$ is compact and at least as close to $\phi$ as $\phi_\comp$ is, it must be the case that $\phi_\comp$ is real-valued.
Since $\min \{ \phi_\comp, b\}$ is compact and at least as close to $\phi$ as $\phi_\comp$ is, we must have $\phi_\comp \le b$. A similar argument proves that $a\leq \phi_\comp$.
\end{proof}

The next lemma, which realizes an arbitrary bounded sequence as a continuous function evaluated along the orbit of a point in a transitive topological dynamical system, can be seen as a version of the Furstenberg correspondence principle~\cite[Lemma~3.17]{MR603625}.
In fact, it allows one to realize a countable
collection of bounded sequences with the help of the same transitive topological dynamical system; in this strengthened form it will contribute to the proof of \cref{thm_pointwise} below.

\begin{lemma}
\label{lemma_metric-correcpondence}
Let $J$ be a finite or countably infinite set and let $\{a_i:i\in J\}$ be a collection of bounded functions from $\N$ to $\C$. Then there exists a compact metric space $X$, a continuous map $\smult\colon X\to X$, functions $F_i\in \cont(X)$ for each $i\in J$, and a point $x\in X$ with a dense orbit under $\smult$ such that
\begin{equation}
\label{eqn_metric_correspondence_for_N}
a_i(n)=  F_i(\smult^n x)\qquad\forall n\in\N,~\forall i\in J.
\end{equation}
\end{lemma}
\begin{proof}
Let $D_i \subset \C$ be a compact set containing the image of $a_i$.
The space
\[
Y \coloneqq \prod_{i\in J} D_i^{\N\cup\{0\}}
\]
is a countable product of compact metric spaces and therefore a compact metric space itself.
We can identify $Y$ with the collection of all sequences $y\colon J\times(\N\cup\{0\}) \to\C$ that satisfy $y(i,n)\in D_i$ for all $n\in\N\cup\{0\}$ and $i\in J$.

Given a point $y\in Y$ we define $\smult (y)$ as
\[
(\smult y)(i,n)=y(i,n+1)
\]
which gives a continuous map $\smult : Y \to Y$.
Let $x$ be the point $x(i,n)\coloneqq a_i(n)$ and let $X$ be the orbit closure of $x$ under the action of $\smult$.
Then $X$ is a compact metric space. Moreover, if we define $F_i(y)\coloneqq y(i,0)$ then \eqref{eqn_metric_correspondence_for_N} is satisfied.
\end{proof}

We are finally ready to prove \cref{thm_jdlgSplitting}.
\begin{proof}
[Proof of \cref{thm_jdlgSplitting}]
We will first deal with the case where $f\in\lp{2}(\N,\Phi)$ is bounded and then derive from it the general case.

Using \cref{lemma_metric-correcpondence} we can find a compact metric space $X$, a continuous map $\smult\colon X\to X$, a function $F\in \cont(X)$ and a point $x\in X$ with a dense orbit under $\smult$ such that $F(\smult^n(x))=f(n)$ for all $n\in\N$. Since $X$ is a compact metric space, we can find (using eg.\ \cite[Theorem~A.4]{MR1958753}) a subsequence $\Psi$ of $\Phi$ such that the measures
\[
\mu_N\coloneqq\frac{1}{|\Psi_N|}\sum_{n\in\Psi_N}\delta_{\smult^nx}
\]
weak$^*$ converge to an $\smult$ invariant Borel probability measure $\mu$ on $X$.
We therefore have a measure preserving system $(X,\mu,\smult)$.
The transformation $\smult$ induces an isometry $\umult$ on the Hilbert space $\lp{2}(X,\mu)$ via $\umult(H) = H \circ\smult$ for all $H\in \lp{2}(X,\mu)$.
Let $F=F_\comp+F_\wm$ be the Jacobs--de Leeuw--Glicksberg decomposition of $F$ given by \cref{thm_JdLG_for_N}.

Next for each $j\in\N$, let $H_j\in \cont(X)$ be such that $\|F_\comp-H_j\|_\mu < 1/j$. 
Let $h_j(n)=H_j(\smult^{n}x)$ for all $n \in \N$ and observe that
\begin{align*}
\|h_j-h_\ell\|^2_\Psi
&
=
\limsup_{N\to\infty}\frac1{|\Psi_N|}\sum_{n\in\Psi_N} \big|H_j(\smult^{n}x)-H_\ell(\smult^{n}x)\big|^2
\\
&
=
\int_X |H_j-H_\ell|^2 \d\mu
=
\|H_j-H_\ell\|_\mu^2,
\end{align*}
which implies, in particular, that $j\mapsto h_j$ is a Cauchy sequence in $\lp{2}(\N,\Psi)$.
Using \cref{prop_weak-completeness}, after refining $\Psi$ if necessary, we can find a function $f_\comp\in\lp{2}(\N,\Psi)$ such that $\|h_j-f_c\|_{\Psi}\to0$ as $j\to\infty$.
We also define $f_\wm$ to be $f-f_\comp$.

To show that $f_\comp$ is compact along $\Psi$, fix $\epsilon > 0$ and let $K\in\N$ be such that
\[
\min\big\{\|\smult^mF_\comp-\smult^kF_\comp\|_\mu : 1 \leq k\leq K \big\} < \epsilon
\]
for every $m\in\N$.
Then, taking $j > 1/\epsilon$ large enough so that $\|h_j-f_\comp\|_\Psi < \epsilon$, we have
\begin{align*}
\|\rmult^mf_\comp-\rmult^kf_\comp\|_\Psi
&
\leq
\|\rmult^mh_j-\rmult^kh_j\|_\Psi + 2\epsilon
\\
&
=
\|\smult^mH_j-\smult^kH_j\|_\mu + 2\epsilon
\\
&
\le
\|\smult^mF_\comp-\smult^k F_\comp\|_\mu+4\epsilon,
\end{align*}
and hence $\min\big\{\|\rmult^mf_\comp-\rmult^kf_\comp\|_\Psi:1\leq k\leq K\big\}<5\epsilon$.
If $f$ takes values in $[a,b]$ then so does $F$. By \cref{cor_compact_part_positive_bounded} it follows that $F_\comp$ also takes values in $[a,b]$.
In this case, we can choose $H_j$ to take values in $[a,b]$ and hence $h_j$ takes values in $[a,b]$ for every $j\in\N$. Finally, since $f_\comp$ is the limit of $h_j$ as $j\to\infty$, we have from \cref{prop_weak-completeness} that it takes values in $[a,b]$ too.

To prove that $f_\wm$ is weak mixing along $\Psi$, let $h\colon\N\to\C$ be bounded and let $\Psi'$ be a \Folner{} subsequence of $\Psi$ such that the correlations $\bilin{\rmult^n f}{h}_{\Psi'}$ exist for every $n\in\N$.
Using \cref{lemma_metric-correcpondence} again, we can find another compact metric space $\tilde{X}$, a continuous map $\tilde{\smult}\colon \tilde{X}\to \tilde{X}$, a function $\tilde{F}\in \cont(\tilde{X})$ and a point $\tilde{x}\in \tilde{X}$ with a dense orbit under $\smult$ such that $\tilde{F}(\tilde{\smult}^n(\tilde{x}))=h(n)$ for all $n\in\N$.

Let $Z\subset X\times \tilde{X}$ be the orbit closure of $(x,\tilde{x})$ under $\smult\times\tilde{\smult}$.
Since $Z$ is a compact metric space, we can find a subsequence $\Psi''$ of $\Psi'$ such that the measures
\[
\nu_N\coloneqq\frac{1}{|\Psi''_N|}\sum_{n\in\Psi_N''} \delta_{(\smult\times\tilde{\smult})^n(x,\tilde{x})}
\]
converge in the weak$^*$ topology to an invariant probability measure $\nu$ on $Z$.
For all $\epsilon>0$, if $j$ is sufficiently large, then
\begin{align*}
\big|\langle \rmult^m f_\wm,h\rangle_{\Psi'}\big|
&
\leq
\left|\langle \rmult^m(f-h_j),h\rangle_{\Psi''}\right|+\epsilon
\\
&
=
\left|\lim_{N\to\infty}\frac1{|\Psi''_N|}\sum_{n\in\Psi''_N} (f-h_j)(n+m)\overline{h(n)}\right|+\epsilon
\\
&
=
\left|\lim_{N\to\infty}\frac1{|\Psi''_N|}\sum_{n\in\Psi''_N} (F-H_j)(\smult^{n+m} x)\overline{\tilde{F}(\tilde{\smult}^n\tilde{x})}\right|+\epsilon
\\
&
=
\left|\int_Z(\smult\times\tilde{\smult})^m\big((F-H_j)\otimes 1\big)\overline{(1\otimes \tilde{F})}\d\nu\right|+\epsilon
\\
&
\leq
\left|\int_Z (\smult\times\tilde{\smult})^m(F_\wm\otimes 1) (1\otimes \tilde{F})\d\nu\right|+2\epsilon.
\end{align*}
For every $\phi \in \cont(X)$ and every $\psi \in \cont(\tilde{X})$ we have
\[
|\bilin{F_\wm \otimes 1}{\phi \otimes \psi}_\nu|
\le
|\bilin{F_\wm}{\phi}_\mu| \, \sup_{z\in\tilde X}\big|\psi(z)\big|
\]
which implies $F_\wm \otimes 1$ in $\lp{2}(Z,\nu)$ is a weak mixing function.
This implies that the set
\[
\left\{ n \in \N : \left|\int_Z (\smult\times\tilde{\smult})^m(F_\wm\otimes 1) (1\otimes \tilde{F})\d\nu\right|>\epsilon \right\}
\]
has zero density with respect to every \Folner{} sequence.
Hence the set
\[
\Big\{ n \in \N : \big|\bilin{\rmult^n f_\wm}{h}_\Psi \big|>3\epsilon \Big\}
\]
has zero density with respect to every \Folner{} sequence, finishing the proof in the case $f$ is bounded.

Next, we deal with the case where $f$ is approximable in $\lp{2}(\N,\Phi)$ by bounded functions.
Suppose for $f\in\lp{2}(\N,\Phi)$ that there is $j \mapsto f_j$ a sequence of bounded functions such that $\|f-f_j\|_\Phi\to 0$ as $j\to\infty$.
Define $\Psi^{(0)}\coloneqq \Phi$.
For every $j\in\N$, apply the decomposition to $f_j$ to obtain a \Folner{} sequence $\Psi^{(j)}$, which is a subsequence of $\Psi^{(j-1)}$, and a decomposition $f_j=f_{j,\comp}+f_{j,\wm}$, where $f_{j,\comp}$ is compact along $\Psi^{(j)}$ and $f_{j,\wm}$ is weak mixing along $\Psi^{(j)}$.

Define $\Psi$ as $\Psi_N\coloneqq \Psi^{(N)}_N$ for all $N\in\N$.
Then, for every $j\in\N$, since $\Psi$ is eventually a \Folner{} subsequence of $\Psi^{(j)}$, the function $f_{j,\comp}$ is compact along $\Psi$ and the function $f_{j,\wm}$ is weak mixing along $\Psi$.
In particular $\bilin{f_{j,\comp}}{f_{\ell,\wm}}_\Psi=0$ for every $j,\ell$ and hence $\|f_j-f_\ell\|_\Psi^2=\|f_{j,\comp}-f_{\ell,\comp}\|_\Psi^2+ \|f_{j,\wm}-f_{\ell,\wm}\|_\Psi^2$.
Since $j \mapsto f_j$ is a Cauchy sequence with respect to $\Phi$ (and hence with respect to $\Psi$), it follows that $j \mapsto f_{j,\comp}$ is also a Cauchy sequence with respect to $\Psi$.
Using \cref{prop_weak-completeness}, and after refining $\Psi$ if needed, we can find a function $f_\comp$ in $\lp{2}(\N,\Psi)$ such that $\|f_{j,\comp}-f_\comp\|_\Psi\to0$ as $j\to\infty$.
It follows that $f_\comp$ is compact with respect to $\Psi$.
Then let $f_\wm=f-f_\comp$ and observe that $\|f_\wm-f_{j,\wm}\|_\Psi\to0$ as $j\to\infty$, which implies that $f_\wm$ is weak mixing.

Finally, we deal with arbitrary $f \in \lp{2}(\N,\Phi)$, which may not be approximable in $\lp{2}(\N,\Phi)$ by bounded functions.
(See \cref{eg:noBoundedApprox} below.)
To begin, pass to a subsequence $\Psi$ of $\Phi$ such that the limit
\[
\lim_{N \to \infty} \frac{1}{|\Psi_N|} \sum_{n \in \Psi_N} |1_{\{ a < |f| \le b \}}(n) f(n)|^2
\]
exists for every $a < b$ in $\N$.
We claim that the sequence $k \mapsto 1_{\{ |f| \le k \}}(n) f(n)$ is Cauchy in $\lp{2}(\N,\Psi)$.
Indeed, if not we can find $\epsilon > 0$ and sequences $a_k,b_k$ in $\N$ with $a_k < b_k < a_{k+1}$ such that
\[
\nbar 1_{\{ |f| \le b_k \}} f - 1_{\{ |f| \le a_k \}} f \nbar^2_\Psi \ge \epsilon
\]
for all $k \in \N$.
But then
\[
\frac{1}{|\Psi_N|} \sum_{n \in \Psi_N} |f(n)|^2
\ge
\sum_{k=1}^K
\frac{1}{|\Psi_N|} \sum_{n \in \Psi_N} |1_{\{ a_k < |f| \le b_k \}}(n) f(n)|^2
\]
for all $K \in \N$ by orthogonality.
Since all limits are assumed to exist we have
\[
\limsup_{N \to \infty}
\frac{1}{|\Psi_N|} \sum_{n \in \Psi_N} |f(n)|^2
\ge
K \epsilon
\]
for all $K \in \N$ contradicting $f \in \lp{2}(\N,\Psi)$.

After passing to a further subsequence the sequence $k \mapsto 1_{\{ |f| \le k \}}(n) f(n)$ has a limit $g$ in $\lp{2}(\N,\Psi)$.
Certainly $g$ is a limit of bounded functions in $\lp{2}(\N,\Psi)$ and therefore has, after once more passing to a subsequence, a splitting into a compact function $g_\comp$ and a weak mixing function $g_\wm$.
Now $f = g_\comp + g_\wm + (f-g)$.
We claim that $f-g$ is weak mixing.
Since the sum of two weak mixing functions is weak mixing, this will conclude the proof.

To prove $f-g$ is weak mixing, fix $h : \N \to \C$ bounded.
After passing to a subsquence depending on $h$ such that inner products exist, we have
\begin{align*}
|\bilin{f - g}{\rmult^n h}_\Psi|
&
\approx
|\bilin{f - 1_{\{ |f| \le k \} } f}{\rmult^n h}_\Psi|
\\
&
=
|\bilin{1_{\{ |f| > k \} } f}{\rmult^n h}_\Psi|
\\
&
\le
\nbar f \nbar_\Psi \nbar 1_{\{ |f| > k \}} \rmult^n h \nbar_\Psi
\le
B \nbar f \nbar_\Psi \sqrt{\upperdens_\Psi(\{ |f| > k \})}
\end{align*}
since $h$ is everywhere bounded by say $B > 0$.
Thus the inner product is small for all $n$ as long as $k$ is large enough by the Markov inequality.
\end{proof}

The following example shows that functions in $\lp{2}(\N,\Phi)$ may not be approximable by bounded functions.

\begin{example}
\label{eg:noBoundedApprox}
Put $\Phi_N = [2^N,2^N + N)$ and
\[
f = \sum_{N \in \N} 1_{\{2^N\}} \sqrt{N}
\]
which is in $\lp{2}(\N,\Phi)$ because
\[
\frac{1}{|\Phi_N|} \sum_{n \in \Phi_N} |f(n)|^2
=
\frac{1}{N} |\sqrt{N}|^2
=
1
\]
but is not approximable in $\lp{2}(\N,\Phi)$ by bounded functions because the sequence $k \mapsto f 1_{\{ f \le k\}}$ converges to zero in $\lp{2}(\N,\Phi)$.
\end{example}

\section{Proof of \texorpdfstring{\cref{thm_hilbertErdosNatural}}{Theorem 2.7}}
\label{sec_finding_p}

In \cref{sec_outlineNatural} we reduced the proof of \cref{thm_erdosNatural} to \cref{thm_hilbertErdosNatural}.
In this section we use the splittings coming from Theorems \ref{thm_besicSplitting} and \ref{thm_jdlgSplitting} of \cref{sec_two_decomps} to finish the proof of \cref{thm_hilbertErdosNatural}.

The main result of this section is the following theorem, which gives us an ultrafilter satisfying several convenient properties.

\begin{theorem}
\label{thm_choosingUltra}
Fix $\epsilon > 0$ and a \Folner{} sequence $\Phi$ on $\N$.
Given $f_\bes\in\bes(\N,\Phi)$ bounded and non-negative, $f_\anti\in\bes(\N,\Phi)^\perp$ bounded and real-valued, and $f_\comp\in\lp{2}(\N,\Phi)$ bounded, non-negative and compact along $\Phi$, one can find a subsequence $\Psi$ of $\Phi$ and an ultrafilter $\ultra{p} \in \beta \N$ such that:
\begin{enumerate}	
[label=\textbf{\textup{U\arabic{enumi}.}},ref=\textbf{U\arabic{enumi}},leftmargin=1cm]
\item
\label{req_essential}
$\upperdens_\Psi(E) > 0$ for all $E \in \ultra{p}$;
\item
\label{req_compact}
$\{ n \in \N : \nbar \rmult^n f_\comp - f_\comp \nbar_\Psi < \tfrac{\epsilon}{3} \} \in \ultra{p}$;
\item
\label{req_besicovitch}
$\nbar \rmult^\ultra{p} f_\bes - f_\bes \nbar_\Psi < \tfrac{\epsilon}{3}$;
\item
\label{req_beig}
$\bilin{f_\comp}{\rmult^\ultra{p} f_\anti}_\Psi$ is non-negative.
\end{enumerate}
\end{theorem}
The proof of \cref{thm_choosingUltra} is given in \cref{subsec_choosingUltra}.
For now we show how, together with the decompositions provided by Theorems \ref{thm_besicSplitting} and \ref{thm_jdlgSplitting}, it implies \cref{thm_hilbertErdosNatural}.

\begin{proof}
[Proof of \cref{thm_hilbertErdosNatural} assuming \cref{thm_choosingUltra}]
Fix a bounded, non-negative function $f\colon \N\to \R$ and a \Folner{} sequence $\Phi$ on $\N$ with $\bilin{1}{f}_\Phi$ existing.
The statement is trivial if $\nbar f\nbar_\Phi=0$, so let us assume that $\nbar f\nbar_\Phi>0$.
Fix also $\epsilon > 0$.
Our goal is to find a subsequence $\Psi$ of $\Phi$ and a non-principal ultrafilter $\ultra{p}\in\beta\N$ such that
\[
\lim_{n \to \ultra{p}} \bilin{ \rmult^n f }{ \rmult^\ultra{p} f }_\Psi
\geq
\bilin{1}{f}_\Psi^2  - \epsilon
\]
holds.

Apply \cref{thm_besicSplitting} and \cref{thm_jdlgSplitting} to obtain, after passing to a subsequence $\Psi$ of $\Phi$, decompositions $f = f_\bes + f_\anti$ and $f = f_\comp + f_\wm$.
Since $f$ is bounded and non-negative, according to the second part of \cref{thm_jdlgSplitting}, the function $f_\comp$ is also bounded and non-negative.
Similarly, $f_\bes$ is bounded and real-valued as well. Since $f_\anti=f-f_\comp$, it also follows that $f_\anti$ is bounded and real-valued, which is another fact that we will use later in the proof.
In fact, after passing to a subsequence of $\Psi$ if necessary, all of $\nbar f_\comp \nbar_\Psi$, $\nbar f_\bes \nbar_\Psi$, $\nbar f_\wm \nbar_\Psi$ and $\nbar f_\anti \nbar_\Psi$ are at most $\nbar f \nbar_\Psi$ by orthogonality and the Pythagoras theorem.

Next we can apply \cref{thm_choosingUltra} with $\epsilon/\nbar f \nbar_\Phi$ in place of $\epsilon$ to get a finer subsequence $\Psi$ and an ultrafilter $\ultra{p}$ satisfying \ref{req_essential} through \ref{req_beig} with $\epsilon/\nbar f \nbar_\Phi$ in place of $\epsilon$.
Finally, pass once more to a subsequence of $\Psi$ such that the inner products $\bilin{f_\comp}{f_\bes}_\Psi$,
$\bilin{\rmult^n f_\wm}{\rmult^\ultra{p} f}_\Psi$, $\bilin{\rmult^n f_\comp}{\rmult^\ultra{p} f_\bes}_\Psi$ and $\bilin{\rmult^n f_\comp}{\rmult^\ultra{p} f_\anti}_\Psi$ exist for all $n \in \N\cup\{0\}$.
Note that $\rmult^\ultra{p} f_\bes$ and $\rmult^\ultra{p} f_\anti$ are well defined since $f_\bes$ and $f_\anti$ are bounded.

We then have
\[
\bilin{\rmult^n f}{\rmult^\ultra{p} f}_\Psi
=
\bilin{\rmult^n f_\wm}{\rmult^\ultra{p} f}_\Psi
+
\bilin{\rmult^n f_\comp}{\rmult^\ultra{p} f_\bes}_\Psi
+
\bilin{\rmult^n f_\comp}{\rmult^\ultra{p} f_\anti}_\Psi
\]
for all $n \in \N$.
We claim that
\begin{align}
\label{eqn_wm_part}
\lim_{n \to \ultra{p}} \bilin{ \rmult^n f_\wm }{ \rmult^\ultra{p} f}_\Psi
&
=
0
\\
\label{eqn_bes_part}
\lim_{n \to \ultra{p}} \bilin{\rmult^n f_\comp}{\rmult^\ultra{p} f_\anti}_\Psi
&
\geq
-\tfrac{\epsilon}{3}
\\
\label{eqn_besmixing_part}
\lim_{n \to \ultra{p}}
\bilin{\rmult^n f_\comp}{\rmult^\ultra{p} f_\bes}_\Psi
&
\geq
\bilin{1}{f}_\Psi^2 - \tfrac{2\epsilon}{3}
\end{align}
are all true for our choice of $\ultra{p}$.
Once \eqref{eqn_wm_part}, \eqref{eqn_bes_part} and \eqref{eqn_besmixing_part} have been established, \eqref{eqn_hilbertErdosNatural} follows immediately and the proof is complete.

Let us first show \eqref{eqn_wm_part}.
Since $f_\wm$ is weak mixing along $\Psi$, we have, for every $\delta > 0$, that the set $\{n\in\N: |\bilin{\rmult^n f_\wm}{\rmult^\ultra{p} f}_\Psi|\geq \delta\}$ has zero density with respect to $\Psi$.
It therefore does not belong to $\ultra{p}$ by \ref{req_essential}.
It follows that $\{n\in\N: |\bilin{\rmult^n f_\wm}{\rmult^\ultra{p} f}_\Psi|<\delta \}$ belongs to $\ultra{p}$ for all $\delta>0$ giving \eqref{eqn_wm_part}.

For the proof of \eqref{eqn_bes_part} note that
\begin{equation}
\label{eqn_compAnti}
\lim_{n \to \ultra{p}} \bilin{\rmult^n f_\comp}{\rmult^\ultra{p} f_\anti}_\Psi
\geq
\bilin{ f_\comp}{\rmult^\ultra{p} f_\anti}_\Psi - \tfrac{\epsilon}{3}
\end{equation}
in light of \ref{req_compact} because of
\[
| \bilin{\rmult^n f_\comp - f_\comp}{\rmult^\ultra{p} f_\anti}_\Psi |
\le
\nbar \rmult^n f_\comp - f_\comp \nbar_\Psi \nbar f_\anti \nbar_\Psi
\le
\nbar \rmult^n f_\comp - f_\comp \nbar_\Psi \nbar f \nbar_\Phi
\]
by Cauchy-Schwarz.
Thus \eqref{eqn_bes_part} follows from \eqref{eqn_compAnti} and \ref{req_beig}.

Utilizing \ref{req_compact} once more this time combined with
\[
| \bilin{\rmult^n f_\comp - f_\comp}{\rmult^\ultra{p} f_\bes}_\Psi |
\le
\nbar \rmult^n f_\comp - f_\comp \nbar_\Psi \nbar f_\bes \nbar_\Psi
\le
\nbar \rmult^n f_\comp - f_\comp \nbar_\Psi \nbar f \nbar_\Phi
\]
via a similar application of Cauchy-Schwarz, we see that
\begin{equation}
\label{eqn_compact_bes_estimate}
\bilin{f_\comp}{\rmult^\ultra{p} f_\bes}_\Psi
\geq
\bilin{1}{f}_\Psi^2-\tfrac{\epsilon}{3}
\end{equation}
implies \eqref{eqn_besmixing_part}.
To prove \eqref{eqn_compact_bes_estimate} use \ref{req_besicovitch} and Cauchy-Schwarz once more to establish
\[
\bilin{f_\comp}{\rmult^\ultra{p} f_\bes}_\Psi \ge \bilin{f_\comp}{f_\bes}_\Psi - \tfrac{\epsilon}{3}
\]
and then we observe that $\bilin{f_\comp}{f_\bes}_\Psi = \|f_\bes\|^2_\Psi+\bilin{f_\comp-f_\bes}{f_\bes}_\Psi$.
Since $f_\comp-f_\bes=f_\anti-f_\wm$ and every weak mixing function belongs to $\bes(\N,\Psi)^\perp$, it follows that $\bilin{f_\comp-f_\bes}{f_\bes}_\Psi=\bilin{f_\anti-f_\wm}{f_\bes}_\Psi=0$ and hence $\bilin{f_\comp}{f_\bes}_\Psi = \|f_\bes\|^2_\Psi$.
Finally, we apply the Cauchy-Schwarz inequality to deduce that $\|f_\bes\|^2_\Psi\geq\bilin{1}{f_\bes}_\Psi^2$ and, using $\bilin{1}{f_\anti}_\Psi=0$, we get
$\bilin{1}{f_\bes}_\Psi^2=\bilin{1}{f}_\Psi^2$.
This implies \eqref{eqn_compact_bes_estimate} and finishes the proof.
\end{proof}

\subsection{Proof of \texorpdfstring{\cref{thm_choosingUltra}}{Theorem 4.1}}
\label{subsec_choosingUltra}

We begin with some preparatory definitions.
\begin{definition}
\label{def_essential_ultrafilter}
Given a \Folner{} sequence $\Phi$ on $\N$ we say an ultrafilter $\ultra{p}$ is $\Phi$ \define{essential} if $\upperdens_\Phi(E) > 0$ for every $E \in \ultra{p}$.
Write $\ess(\Phi)$ for the set of $\Phi$ essential ultrafilters on $\N$.
\end{definition}

Observe that property \ref{req_essential} in \cref{thm_choosingUltra} means exactly that $\ultra{p}$ is a $\Psi$ essential ultrafilter.

Recall from \cref{sec_outlineNatural} the definition of $\M{\Phi}$.

\begin{definition}
\label{def_almost_everywhere}
A Borel measurable property of ultrafilters is said to hold $\Phi$ \define{almost everywhere} if the set of ultrafilters $\ultra{p}$ with the property has full measure with respect to every $\mu \in \M{\Phi}$.
\end{definition}

\begin{lemma}\label{lemma_support}
Let $\Phi$ be a \Folner{} sequence on $\N$.
Then $\Phi$ almost every $\ultra{p}$ belongs to $\ess(\Phi)$.
\end{lemma}
\begin{proof}
First, observe that
\[
\ess(\Phi)=\bigcap_{E\subset\N:\upperdens_\Phi(E)=0}\cl(\N\setminus E)=\beta\N\setminus\bigcup_{E\subset\N:\upperdens_\Phi(E)=0}\cl(E)
\]
so that it is a closed set (and hence Borel).
Fix $\mu\in\M\Phi$.
We claim that the support of $\mu$ is contained in $\ess(\Phi)$.
Since $\mu$ is Radon this implies $\mu(\ess(\Phi))=1$ as desired.

To prove the claim, fix $\ultra{p}\in\beta\N\setminus \ess(\Phi)$.
We need to show that there exists an open set $U\subset\beta\N$ containing $\ultra{p}$ such that $\mu(U) = 0$.
But since $\ultra{p}\in\beta\N\setminus \ess(\Phi)$, there exists $E\subset\N$ with $\upperdens_\Phi(E)=0$ and $\ultra{p}\in\cl(E)$.
The set $\cl(E)$ is then an open subset of $\beta\N$ containing $\ultra{p}$ and with $\mu(\cl(E))\leq\upperdens_\Phi(E)=0$.
\end{proof}

\begin{definition}
\label{def_bohr_set}
A \define{Bohr set} on $\N$ is any set of the form $a^{-1}(U)$ where $a$ is a homomorphism from $\N$ into a compact metrizable abelian group $K$ and $U$ is a non-empty open subset of $K$ whose topological boundary $\partial U$ has zero Haar measure.
A Bohr set is a \define{Bohr$_0$ set} if $U$ contains the identity element of $K$.
\end{definition}

There are various minor variations on the definition of Bohr sets appearing in the literature.
For example, sometimes authors restrict attention to the case where $K$ is a product of finitely many copies of the circle group and $U$ is a product of arcs.
Alternatively, one could define Bohr sets and Bohr$_0$ sets with the help of the Bohr topology on the integers, which is the topology induced by the embedding of $\Z$ into its Bohr compactification (cf.\ \cite{MR0695335}, \cite[Section 1]{MR2249261} and \cite{MR2755924}).
\cref{def_bohr_set} is the most convenient for our needs because with it the following lemmas are straightforward to prove.

\begin{lemma}
\label{lem_bohr_intersect}
If $A$ and $B$ are Bohr$_0$ sets then so is $A \cap B$.
\end{lemma}
\begin{proof}
Write $A = a^{-1}(U)$ and $B = b^{-1}(V)$ where $a : \N \to K$ and $b : \N \to L$ are homomorphisms to compact metrizable topological groups $K$ and $L$ respectively.
Then $A \cap B = c^{-1}(U \times V)$ where $c : \N \to K \times L$ is the homomorphism $c(n) = (a(n),b(n))$.
\end{proof}

The following lemma is folklore; we reproduce a short proof from \cite[Lemma 2.7]{GKR}.
\begin{lemma}
\label{lem_bohr_semigroup}
Let $a : \N \to G$ be a homomorphism from $\N$ to a compact abelian topological group $G$.
Then the closure of the image of $a$ is a subgroup of $G$.	
\end{lemma}
\begin{proof}
Define $S\coloneqq \overline{\{a(n):n\in\N\}}$ and
\[
H\coloneqq \overline{\{a(n):n\in\N\}}\cup \{0\}\cup \overline{\{-a(n):n\in\N\}}.
\]
We have to show that $S=H$.
Define $A\coloneqq \bigcap_{N\in\N}\overline{\{a(n): n\geq N\}}$.
Since $A$ is the intersection of a nested family of non-empty compact sets, it is non-empty.
Pick any $x\in A$.
Since $A$ is $H$-invariant, we have $H+x\subset A$ and hence $A=H$.
But $A\subset S$, which now implies $H\subset S$.
\end{proof}

\begin{lemma}
\label{lem_bohr_sets_have_positive_density}
If $B\subset \N$ is a non-empty Bohr set then for every \Folner{} sequence $\Phi$  its indicator function $1_B$ is in $\bes(\N,\Phi)$ and $\dens_\Phi(B) > 0$.
\end{lemma}
\begin{proof}
Let $K$ be a compact abelian group, let $a\colon\N\to K$ be a homomorphism and let $U\subset K$ be an open set with zero measure boundary and such that $B=a^{-1}(U)$.
Replacing $K$ with the closure $\overline{a(\N)}$ we can assume that $a$ has a dense image.

For each $N\in\N$, let $\mu_N$ be the probability measure on $K$ obtained as the average of the Dirac point masses at the points $\{a(n):n\in\Phi_N\}$.
Since $\Phi$ is a F\o lner sequence, any weak$^*$ limit point $\mu$ of $(\mu_N)_{N\in\N}$ is invariant under $a(\N)$.
By \cref{lem_bohr_semigroup} it follows that $\mu=\haar$ is the Haar measure on $K$.
Since $U$ is open we have
\[
0
<
\haar(U)
=
\lim_{N\to\infty}\mu_N(U)
=
\lim_{N\to\infty}\frac{|B\cap \Phi_N|}{|\Phi_N|}
=
\dens_\Phi(B)
\]
in view of \cite[Theorem~A.5]{MR1958753}.
Finally, since finite linear combinations of characters (i.e., continuous homomorphisms from $K$ to the circle group $S^1$) are dense in $\lp{2}(K,\haar)$, we can find for every $\epsilon>0$ a linear combination $f$ of characters such that $\|f-1_U\|_{\haar}<\epsilon$.
Since $\mu=\haar$ and $f-1_U$ is $\haar$-almost everywhere continuous, it follows that $\|f\circ a-1_B\|_\Phi=\|f-1_U\|_{\haar}<\epsilon$.
Since $f\circ a$ is a trigonometric polynomial and $\epsilon$ was arbitrary, we conclude that $1_B\in\bes(\N,\Phi)$.
\end{proof}

\begin{lemma}
\label{cor_compact_returns_Bohralternative}
For every function $f\in\lp{2}(\N,\Phi)$ that is compact along $\Phi$ and every $\epsilon>0$, the set $\{ n \in \N : \nbar \rmult^n f - f \nbar_\Phi < \epsilon \}$ contains a Bohr$_0$ set $B_\comp$.
\end{lemma}
\begin{proof}
Let $g(n)=\nbar \rmult^{|n|} f - f \nbar_\Phi$ for every $n\in\Z$.
Since $f$ is compact along $\Phi$ it follows that the closure $\Omega$ of the set $\{ \rmult^k g : k \in \Z \}$ has a finite $\epsilon$-dense subset with respect to the uniform metric for every $\epsilon > 0$.
It therefore has compact closure.

We can make $\Omega$ into a compact topological group by defining
\[
(\rmult^n g) \star (\rmult^k g) = \rmult^{k+n} g
\]
for all $n,k \in \Z$ and extending $\star$ to a binary operation on all of $\Omega$ by continuity.
Define $U_\eta\coloneqq \{\phi:\Z\to [0,\infty): \phi(0)<\eta\}$.
Using the homomorphism $a(n) = \rmult^n g$ from $\N$ to our topological group $(\Omega, \star)$ we see that $\{ n \in \N : \nbar \rmult^n f - f \nbar_\Phi < \epsilon \} = \{n\in\N: a(n)\in U_\epsilon\}$.
Moreover, $\{n\in\N: a(n)\in U_\eta\}\subset \{ n \in \N : \nbar \rmult^n f - f \nbar_\Phi < \epsilon \} $ for every $\eta<\epsilon$.
Since Haar measure on $\Omega$ is finite and the boundaries of the sets $U_\eta$ are pairwise disjoint, for all but countably many $\eta > 0$ the boundary of the set $U_\eta$ has zero Haar measure.
Pick any $\eta<\epsilon$ for which $\partial U_\eta$ has measure $0$ and let $B_\comp\coloneqq\{n\in\N: a(n)\in U_\eta\}$.
\end{proof}

The following two theorems, proved in subsequent subsections, will be used in the proof of \cref{thm_choosingUltra}.
The first, which will be used to guarantee \ref{req_besicovitch}, relies on the pointwise ergodic theorem.
Its proof can be found in \cref{sec_pointwise}.

\begin{theorem}
\label{thm_Besicovitchpshift}
Let $\Phi$ be a \Folner{} sequence on $\N$ and let $f\in\bes(\N,\Phi)$.
For every $\epsilon>0$ there exists a Bohr$_0$ set $B$ and a subsequence $\Psi$ of $\Phi$ such that for $\Psi$ almost every ultrafilter $\ultra{p}\in\cl(B)$ we have $\|\rmult^\ultra{p} f-f\|_\Psi<\epsilon$.
\end{theorem}

The second is a modification of an argument due to \Beiglbock{}~\cite[Lemma~2]{Beiglbock11} and will be used to guarantee \ref{req_beig}.
Its proof is given in \cref{sec_beig}.

\begin{theorem}
\label{thm_beig}
Suppose $f$ is a real-valued bounded function that belongs to $\bes(\N,\Psi)^\perp$. Then for every non-empty Bohr set $B \subset \N$ and every bounded function $h \colon \N \to \R$ the set
\begin{equation}
\label{eqn_set_of_good_ultrafilters_in_N}
\left\{ \ultra{p} \in \cl(B) : \limsup_{N \to \infty} \frac{1}{|\Psi_N|} \sum_{m \in \Psi_N} h(m) \, (\rmult^\ultra{p} f)(m) \ge 0 \right\}
\end{equation}
is Borel measurable and has positive measure with respect to every $\mu \in \M{\Psi}$.
\end{theorem}

With these theorems we can give the proof of \cref{thm_choosingUltra}.

\begin{proof}
[Proof of \cref{thm_choosingUltra}]
Fix $\epsilon > 0$ and a \Folner{} sequence $\Phi$ on $\N$ along which $f_\comp \in \lp{2}(\N,\Phi)$ is compact, $f_\bes \in \bes(\N,\Phi)$, and $f_\anti \in \bes(\N,\Phi)^\perp$.
We need to find a subsequence $\Psi$ of $\Phi$ and an ultrafilter $\ultra{p}$ such that \ref{req_essential} through \ref{req_beig} are satisfied.

\cref{cor_compact_returns_Bohralternative} gives that there exists a Bohr$_0$ set $B_\comp$ contained in $\{ n \in \N : \nbar \rmult^n f_\comp - f_\comp \nbar_\Phi < \tfrac{\epsilon}{3} \}$.
\cref{thm_Besicovitchpshift} implies that, passing to a subsequence $\Psi$ of $\Phi$, there exists a Bohr$_0$ set $B_\bes$ such that for $\Psi$ almost every $\ultra{p}\in\cl(B_\bes)$ we have $\|\rmult^\ultra{p}f_\bes-f_\bes\|_\Psi<\epsilon/3$.
The set $B\coloneqq B_\comp \cap B_\bes$ is a Bohr$_0$ set by \cref{lem_bohr_intersect}.
Note that $\Psi$ almost every $\ultra{p}\in\cl(B)$ satisfies \ref{req_compact} and \ref{req_besicovitch}.
Applying \cref{thm_beig} with $f = f_\anti$ and $h = f_\comp$ we deduce that the set
\begin{equation}
\label{eqn_proof_thm_choosingUltra}
\left\{ \ultra{p} \in \cl(B) : \limsup_{N \to \infty} \frac{1}{|\Psi_N|} \sum_{m \in \Psi_N} f_\comp(m) \, (\rmult^\ultra{p} f_\anti(m)) \ge 0 \right\}
\end{equation}
 has positive measure for any $\mu\in\M\Psi$.
Notice that any $\ultra{p}$ in the set \eqref{eqn_proof_thm_choosingUltra} satisfies \ref{req_beig}.
Since any such $\ultra{p}$ belongs to $\cl(B)$ it follows that $\Psi$ almost every $\ultra{p}$ in the set \eqref{eqn_proof_thm_choosingUltra} satisfies \ref{req_compact}, \ref{req_besicovitch} and \ref{req_beig}.

Finally, in view of \cref{lemma_support}, $\Psi$ almost every $\ultra{p}\in\beta\N$ satisfies \ref{req_essential}.
This means that $\Psi$ almost every $\ultra{p}$ in the set \eqref{eqn_proof_thm_choosingUltra} satisfies \ref{req_essential}, \ref{req_compact}, \ref{req_besicovitch}, and \ref{req_beig}.
\end{proof}

\subsection{Proof of \texorpdfstring{\cref{thm_Besicovitchpshift}}{Theorem 4.10}}
\label{sec_pointwise}

In this section we present a proof of \cref{thm_Besicovitchpshift}.
We start with the following lemma.

\begin{lemma}
\label{lem_ultraBohrIsBohr}
Let $\Phi$ be a \Folner{} sequence on $\N$.
If $a $ is a trigonometric polynomial and $\ultra{p} \in \beta \N$ then $\rmult^\ultra{p} a$ is a trigonometric polynomial and $\|\rmult^\ultra{p} a\|_\Phi=\|a\|_\Phi$.
\end{lemma}
\begin{proof}
Choose $c_1,\dots,c_J \in \C$ and $\theta_1,\dots,\theta_J \in\R$ such that $a$ has the form \eqref{eqn:trigPoly}.
Define $d_j\coloneqq \lim_{m\to\ultra{p}} c_j e^{2\pi  i \theta_j m}$.
Notice that
\[
(\rmult^{\ultra{p}}a)(n) = \sum_{j=1}^J d_j e^{2\pi  i \theta_j n}
\]
and, since $|c_j|=|d_j|$, it follows from \cref{lemma_exponentialfolnervanish} that $\|\rmult^\ultra{p} a\|_\Phi=\|a\|_\Phi$.
\end{proof}

We will also need a version of the pointwise ergodic theorem.
There are \Folner{} sequences for which the pointwise ergodic theorem does not hold \cite{MR0360999}.
However, every \Folner{} sequence has a subsequence along which the pointwise ergodic theorem holds.

\begin{definition}
A \Folner{} sequence $\Phi$ is called \define{tempered} if there exists $C>0$ such that
\[
\left|\bigcup_{k=1}^N \Phi_{N+1} - \Phi_k \right|
\leq
C|\Phi_{N+1}|
\]
for every $N \in \N$, where $\Phi_{N+1} - \Phi_k$ is the set of differences.
\end{definition}

According to \cite[Proposition 1.4]{MR1865397}, every \Folner{} sequence has a tempered subsequence.
Here is the pointwise ergodic theorem for tempered \Folner{} sequences.

\begin{theorem}[see {\cite[Theorem 1.2]{MR1865397}}]
\label{thm_pointwiseergodictheorem}
Let $(X,\nu,T)$ be a measure preserving system and let $\Phi$ be a tempered \Folner{} sequence.
Then for every $f\in \lp{1}(X,\nu)$ the limit
\[
\lim_{N\to\infty} \frac{1}{|\Phi_N|} \sum_{n \in \Phi_N}f(T^n x)
\]
exists for $\nu$ almost every $x \in X$ and defines a $T$ invariant function in $\lp{1}(X,\mu)$.
\end{theorem}

\begin{theorem}
\label{thm_pointwise}
Let $\Phi$ be a \Folner{} sequence on $\N$ and let $h\in\bes(\N,\Phi)$ be bounded.
Then there is a subsequence $\Psi$ of $\Phi$ with $\|\rmult^\ultra{p} h\|_\Psi=\|h\|_{\Psi}$ for $\Psi$ almost every $\ultra{p}$.
\end{theorem}
\begin{proof}
First we pass to a tempered subsequence $\Psi$ of $\Phi$.
Let $j \mapsto a_j$ be a sequence of trigonometric polynomials such that $\| h-a_j\|_\Psi \to 0$ as $j \to \infty$.
Apply \cref{lemma_metric-correcpondence} to the collection $\{ h, a_1, a_2, \dots \}$ to find a compact metric space $X$, a continuous map $\smult \colon X \to X$, a point $x\in X$ with a dense orbit under $\smult$ and functions $H,F_1,F_2,\dots$ in $\cont(X)$ such that $a_j(n)=F_j(\smult^n x)$ and $h(n) = H(\smult^n x)$ for all $j,n \in \N$.

For each $\ultra{p} \in \beta \N$ define the map $\smult^\ultra{p} \colon X \to X$ by
\[
\smult^\ultra{p} x = \lim_{n\to \ultra{p}} \smult^n x
\]
and notice that
\begin{equation}
\label{eq_lemma_lem_broken-BAP-remains-BAP}
(\rmult^\ultra{p} a_j)(n)
=
\lim_{m\to \ultra{p}}a_j(n+m)
=
\lim_{m\to \ultra{p}} F_j \left( \smult^n \smult^{m}x \right)
=
F_j(\smult^n \smult^\ultra{p} x)
\end{equation}
for every $j,n\in\N$ and every $\ultra{p} \in \beta \N$.
We similarly have
\begin{equation}
\label{eq_lemma_lem_broken-BAP-remains-BAP2}
(\rmult^\ultra{p} h)(n) = H(\smult^n \smult^\ultra{p} x)
\end{equation}
for all $n \in \N$ and every $\ultra{p} \in \beta \N$.

The map $\pi \colon \beta\N \to X$ defined by $\ultra{p}\mapsto \smult^\ultra{p} x$ is continuous and surjective by the universal property of $\beta \N$ and the fact that $\{ \smult^n x : n \in \N \}$ is dense in $X$ respectively.

We next wish to prove that
\begin{equation}
\label{eqn:pointwise_average_constant}
\lim_{N \to \infty} \frac{1}{|\Psi_N|} \sum_{n \in \Psi_N} |F_j(\smult^n y)|^2
=
\nbar a_j \nbar_\Psi^2
\end{equation}
for all $y \in X$ and all $j \in \N$.
Fix $y \in X$ and $j \in \N$.
Since $\pi$ is surjective there is $\ultra{p} \in \beta \N$ with $\smult^\ultra{p} x = y$.
We then have
\[
\frac{1}{|\Psi_N|} \sum_{n \in \Psi_N} | F_j (\smult^n y)|^2
=
\frac{1}{|\Psi_N|} \sum_{n \in \Psi_N} | (\rmult^\ultra{p} a_j) (n)|^2
\]
from \eqref{eq_lemma_lem_broken-BAP-remains-BAP}.
By \cref{lem_ultraBohrIsBohr} the function $\rmult^\ultra{p} a_j$ is also a trigonometric polynomial so
\[
\lim_{N \to \infty} \frac{1}{|\Psi_N|} \sum_{n \in \Psi_N} | (\rmult^\ultra{p} a_j) (n)|^2
=
\nbar \rmult^\ultra{p} a_j \nbar_\Psi^2
\]
holds by \cref{lemma_exponentialfolnervanish}.
\cref{lem_ultraBohrIsBohr} also gives $\nbar \rmult^\ultra{p} a_j \nbar_\Psi = \nbar a_j \nbar_\Psi$ establishing \eqref{eqn:pointwise_average_constant}.

Write $\umult$ for the isometry of $\lp{2}(X,\nu)$ defined by $\umult(f) = f \circ \smult$ for all $f \in \lp{2}(X,\mu)$.
By a version of the mean ergodic theorem of von Neumann (cf.~\cite[Theorem~3.33]{MR1958753}) the limit
\[
\lim_{N \to \infty} \frac{1}{|\Psi_N|} \sum_{n \in \Psi_N} \umult^n \left( |F_j|^2 \right)
\]
exists in $\lp{2}(X,\nu)$ for all $j \in \N$ and is equal to the orthogonal projection in $\lp{2}(X,\nu)$ of $|F_j|^2$ onto the closed subspace of $\umult$ invariant functions.
Since constant functions are $\umult$ invariant, the above combined with \eqref{eqn:pointwise_average_constant} implies for all $j \in \N$ that
\[
\int |F_j|^2 \d \nu
=
\nbar a_j \nbar_\Psi^2
\]
is the orthogonal projection in $\lp{2}(X,\nu)$ of $|F_j|^2$ onto the closed subspace of $\umult$ invariant functions.

We are now ready to prove that $\nbar \rmult^\ultra{p} h \nbar_\Psi = \nbar h \nbar_\Psi$ for $\Psi$ almost every $\ultra{p}$.
To this end fix $\mu \in \M{\Psi}$ and let
$\nu = \pi \mu$ for the push-forward of $\mu$ under the map $\pi$.
Since $\mu$ is by definition a weak$^*$ limit point of the set $\{\mu_N:N\in\N\}$, where $\mu_N$ is as in \eqref{eqn_measures_on_betaN}, it follows that $\nu$ is a weak$^*$ accumulation point of the set $\{\pi\mu_N:N\in\N\}$.
Since $X$ is a compact metric space, the space of probability measures on $X$ is metrizable, and hence there exists a subsequence $\Xi$ of $\Psi$ such that
\begin{equation}
\label{eqn:wow_its_generic}
\nu = \lim_{N \to \infty} \frac{1}{|\Xi_N|} \sum_{n \in \Xi_N} \delta_{\smult^n x}
\end{equation}
in the weak$^*$ topology in $X$, where $\delta_{\smult^n x}$ is the point mass on $X$ at the point $\smult^n x$.
We remark that while every measure $\mu\in\M\Psi$ is the limit of a sub-net of $(\mu_N)_{N\in\N}$, there is in general no subsequence of $(\mu_N)_{N\in\N}$ which converges to $\mu$ because $\beta\N$ is not metrizable.

Since the functions $H_j$ and $F$ are continuous on $X$ we may calculate from \eqref{eqn:wow_its_generic} that
\begin{align*}
\nbar F_j - H \nbar_\nu^2
&
=
\lim_{N \to \infty} \frac{1}{|\Xi_N|} \sum_{n \in \Xi_N} |F_j(\smult^n x) - H(\smult^n x) |^2
\\
&
=
\lim_{N \to \infty} \frac{1}{|\Xi_N|} \sum_{n \in \Xi_N} |a_j(n) - h(n) |^2
=
\nbar a_j - h \nbar_\Psi^2
\end{align*}
for all $j \in \N$, with the last equality holding because $h$ and all $a_j$ belong to $\bes(\N,\Psi)$.
The hypothesis that $\|a_j - h\|_\Psi \to 0$ as $j\to\infty$ therefore implies $\|F_j - H\|_\nu \to 0$ as $j \to \infty$.
Since orthogonal projections on Hilbert spaces are continuous we conclude that
\begin{equation}
\label{eqn:bes_pointwise_limit}
\int |H|^2 \d \nu
=
\lim_{j \to \infty} \nbar a_j \nbar_\Psi^2
=
\nbar h \nbar_\Psi^2
\end{equation}
is the orthogonal projection of $|H|^2$ to the closed subspace of $\umult$ invariant functions.

Next, we apply \cref{thm_pointwiseergodictheorem} to deduce that the limit
\[
\lim_{N \to \infty} \frac{1}{|\Psi_N|} \sum_{n \in \Psi_N} |H(\smult^n y)|^2
\]
exists for $\nu$ almost every $y\in X$ and defines a $\umult$ invariant function in $\lp{2}(X,\nu)$.
Since $H$ is bounded, this limit is also bounded.
This limit must therefore be the projection \eqref{eqn:bes_pointwise_limit} of $|H|^2$ to the closed subspace of $\umult$ invariant functions.
In other words
\[
\lim_{N \to \infty} \frac{1}{|\Psi_N|} \sum_{n \in \Psi_N} |H(\smult^n y)|^2
=
\nbar h \nbar_\Psi^2
\]
for $\nu$ almost every $y$.
Finally, since $\nu$ is the push-forward of $\mu$ under $\pi$, it follows from \eqref{eq_lemma_lem_broken-BAP-remains-BAP2} that $\nbar \rmult^\ultra{p} h \nbar_\Phi = \nbar h\nbar_\Phi$ for $\mu$ almost every $\ultra{p} \in \beta \N$.
Since $\mu \in \M{\Psi}$ was arbitrary we are done.

\end{proof}

We are now ready to finish the proof of \cref{thm_Besicovitchpshift}

\begin{proof}[Proof of \cref{thm_Besicovitchpshift}]
Let $\Phi$ be a \Folner{} sequence on $\N$, let $f\in\bes(\N,\Phi)$ and let $\epsilon>0$.
Let $a$ be a trigonometric polynomial such that $\|f-a\|_\Phi<\epsilon/3$.
Notice that $f-a\in\bes(\N,\Phi)$ and hence, using \cref{thm_pointwise}, we can find a subsequence $\Psi$ of $\Phi$ such that for $\Psi$ almost every $\ultra{p}\in\beta\N$
$$\big\|\rmult^\ultra{p}f-f\big\|_\Psi\leq \big\|\rmult^\ultra{p}(f-a)\big\|_\Psi +\big\|\rmult^\ultra{p}a-a\big\|_\Psi+ \big\|a-f\big\|_\Psi \leq \big\|\rmult^\ultra{p}a-a\big\|_\Psi+\frac{2\epsilon}3.$$
It now suffices to find a Bohr$_0$ set $B$ such that for every $\ultra{p} \in \cl(B)$ we have $\big\|\rmult^\ultra{p}a-a\big\|_\Psi\leq\epsilon/3$.

Write $a(n)=\sum_{j=1}^Jc_je^{2\pi in\theta_j}$ for some $c_1,\dots,c_J\in\C$ and $0 \le \theta_1,\dots,\theta_J < 1$.
Let $M=\max_j|c_j|$ and let $\alpha\colon\N\to\T^J$ be the homomorphism $\alpha(n)=(n\theta_1,\dots,n\theta_J)$ (where $\T^J$ is the torus $\R^J/\Z^J$ as usual).
Consider the open set $U=\left(-\frac\epsilon{3MJ},\frac\epsilon{3MJ}\right)^J\subset\T^J$ and let $B=\alpha^{-1}(U)$.
Certainly the boundary of $U$ has zero Haar measure in $\T^J$ so $B$ is a Bohr$_0$ set.
Notice that for every $m\in B$ and every $n\in\N$,
\begin{equation}\label{eq_proof_thm_Besicovitchpshift}
\big|(\rmult^m a)(n)-a(n)\big|=\left|\sum_{j=1}^Jc_je^{2\pi i n\theta_j}\big(e^{2\pi i m\theta_j}-1\big)\right|<\frac\epsilon3
\end{equation}
holds.
Finally, let $\ultra{p}\in\cl(B)$.
In view of \eqref{eq_proof_thm_Besicovitchpshift}, $|(\rmult^\ultra{p}a)(n)-a(n) |<\epsilon/3$ for every $n\in\N$, and therefore also $\big\|\rmult^\ultra{p}a-a\big\|_\Psi\leq\epsilon/3$.
\end{proof}

\subsection{Proof of \texorpdfstring{{\cref{thm_beig}}}{Theorem 4.11}}
\label{sec_beig}

This subsection is devoted to the proof of \cref{thm_beig}. The ideas used in this proof were motivated by the proof of \cite[Lemma~2]{Beiglbock11}.

\begin{proof}[Proof of \cref{thm_beig}]
Let $\mu \in \M{\Psi}$.
Since $B$ is a non-empty Bohr set, we have by \cref{lem_bohr_sets_have_positive_density} that $\dens_\Psi(B)$ exists and is positive. It follows that $\mu\big(\cl(B)\big)=\dens_\Psi(B)>0$.
Define a new probability measure $\mu_B$ on $\beta \N$ by
\[
\mu_B(\Omega) \coloneqq \frac{\mu(\Omega\cap \cl(B))}{\mu(\cl(B))}
\]
for all Borel sets $\Omega\subset\beta \N$.

For each $n\in \N$ the map $\ultra{p} \mapsto (\rmult^\ultra{p} f)(n)=\lim_{m\to \ultra{p}}f(n+m)$ from $\beta \N\to\R$ is continuous, and hence measurable.
Therefore, so is the map
\[
\ultra{p} \mapsto \limsup_{N\to\infty} \frac{1}{|\Psi_N|}\sum_{n \in \Psi_N} h(n) \, (\rmult^\ultra{p} f)(n),
\]
which shows that the set defined in \eqref{eqn_set_of_good_ultrafilters_in_N} is also measurable.
In order to show that the set in \eqref{eqn_set_of_good_ultrafilters_in_N} has positive measure, it suffices to establish the inequality
\[
\int_{\beta\N} \limsup_{N \to \infty} \frac{1}{|\Psi_N|} \sum_{n \in \Psi_N}  h(n) \, (\rmult^\ultra{p} f)(n) \d \mu_B(\ultra{p})
\ge
0.
\]
Using Fatou's lemma it thus suffices to prove that
\begin{equation}
\label{eq_proof_lemma_beiglebock_in_N}
\limsup_{N \to \infty} \frac{1}{|\Psi_N|} \sum_{n \in \Psi_N} h(n)
\int_{\beta\N}  (\rmult^\ultra{p} f)(n) \d\mu_B(\ultra{p})
\ge
0.
\end{equation}
Notice that
\begin{align*}
\Bigg| \int_{\beta\N} (\rmult^\ultra{p} f)(n)\d\mu_B(\ultra{p}) \Bigg|
&
=
\frac1{\mu(\cl(B))}\left|\int_{\beta\N} 1_{\cl(B)}(\ultra{p}) (\rmult^\ultra{p} f)(n)\d\mu(\ultra{p})\right|
\\
&
\leq
\limsup_{N\to\infty}\left|\frac1{|\Psi_N|}\sum_{m\in\Psi_N}1_{B}(m)f(n+m)\right|
\\
&
=
\limsup_{N\to\infty}\left|\frac1{|\Psi_N|}\sum_{m\in\Psi_N}1_{B+n}(m)f(m)\right|.
\end{align*}
Since $f\in\bes(\N,\Psi)^\perp$ and $m\mapsto 1_{B+n}(m)$ is Besicovitch almost periodic along $\Psi$ by \cref{lem_bohr_sets_have_positive_density}, we conclude that
$$
\limsup_{N\to\infty}\left|\frac1{|\Psi_N|}\sum_{m\in\Psi_N}1_{B+n}(m)f(m)\right|=0
$$
and therefore
$$
\Bigg| \int_{\beta\N} (\rmult^\ultra{p} f)(n)\d\mu_B(\ultra{p}) \Bigg|=0$$
for every $n\in\N$.
This implies \eqref{eq_proof_lemma_beiglebock_in_N} and finishes the proof.
\end{proof}

\section{The proof over countable amenable groups}
\label{sec:erdosGroups}

The proof of \cref{thm_erdos_sumset_amenable} is in broad strokes the same as that for $\N$ given in the previous sections.
In this section we discuss the salient differences.

We begin with a discussion of ultrafilters on countable groups.
Just as over $\N$, or any other set, an \define{ultrafilter} on a countable group $G$ is any non-empty family $\ultra{p}$ of non-empty subsets of $G$ that is closed under intersections and supersets, and contains either $A$ or $G \setminus A$ for every $A \subset G$.
For each $g \in G$ the collection $\ultra{p}_g \coloneqq \{ A \subset G : g \in A \}$ is an ultrafilter, called the \define{principal} ultrafilter at $g$.

Denote by $\beta G$ the set of all ultrafilters on $G$.
The sets $\cl(A) = \{ \ultra{p} \in \beta G : A \in \ultra{p} \}$ form a base for a topology on $\beta G$ that is compact and Hausdorff.
Moreover, with this topology $\beta G$ becomes universal for maps $f$ from $G$ to compact, Hausdorff spaces $K$ in the sense that any such map extends to a continuous map $\beta f \colon \beta G \to K$ with $(\beta f)(\ultra{p}_g) = f(g)$ for all $g \in G$.
We usually write
\[
\lim_{g \to \ultra{p}} f(g) \coloneqq (\beta f)(\ultra{p})
\]
for convenience.

Write $A g^{-1} = \{ h \in G : hg \in A \}$ and $g^{-1} A = \{ h \in G : gh \in A \}$ whenever $g \in G$ and $A \subset G$.
Write also $A \ultra{p}^{-1} = \{ g \in G : g^{-1} A \in \ultra{p} \}$ for all $A \subset G$ and all $\ultra{p} \in \beta G$.
With these definitions we have $A g^{-1} = A \ultra{p}_g^{-1}$ for all $g \in G$.
Multiplication on $G$ extends to $\beta G$ in two ways.
For all $\ultra{p},\ultra{q}$ in $\beta G$ both of
\begin{align*}
\ultra{p} \ltimes \ultra{q} &= \{ A \subset G : \{ g \in G : g^{-1} A \in \ultra{q} \} \in  \ultra{p} \} \\
\ultra{p} \rtimes \ultra{q} &= \{ A \subset G : \{ g \in G : A g^{-1} \in \ultra{p} \} \in \ultra{q} \}
\end{align*}
define associative binary operations on $\beta G$.
Using both allows us to generalize \cref{lem:ultrafilterErdos} to countable groups.

\begin{lemma}
\label{lem:ultrafilterErdosGroups}
Fix $A \subset G$.
There are non-principal ultrafilters $\ultra{p}$ and $\ultra{q}$ with the property that $A \in \ultra{p} \ltimes \ultra{q}$ and $A \in \ultra{p} \rtimes \ultra{q}$ if and only if there are infinite sets $B,C \subset G$ with $BC \subset A$.
\end{lemma}
\begin{proof}
First suppose that $BC \subset A$ for infinite sets $B,C \subset G$.
Let $\ultra{p}$ and $\ultra{q}$ be non-principal ultrafilters containing $B$ and $C$ respectively.
For all $c \in C$ we have $B \subset Ac^{-1}$ so $A$ belongs to $\ultra{p} \rtimes \ultra{q}$.
For all $b \in B$ we have $C \subset b^{-1} A$ so $A$ also belongs to $\ultra{p} \ltimes \ultra{q}$.

Conversely, suppose that we can find non-principal ultrafilters $\ultra{p}$ and $\ultra{q}$ with $A$ belonging to both $\ultra{p} \ltimes \ultra{q}$ and $\ultra{q} \rtimes \ultra{q}$.
Thus $\{ g \in G : g^{-1} A \in \ultra{q} \} \in \ultra{p}$ and $\{ g \in G : A g^{-1} \in \ultra{p} \} \in \ultra{q}$.
We construct injective sequences $n \mapsto b_n$ and $n \mapsto c_n$ in $G$ such that $b_i c_j \in A$ for all $i,j \in \N$.
First choose $b_1 \in G$ with $b_1^{-1} A \in \ultra{q}$.
Next, choose $c_1 \in G$ from
\[
b_1^{-1} A \cap \{ y \in G : A y^{-1} \in \ultra{p} \}
\]
which is possible since both sets above belong to $\ultra{q}$.
Next, choose $b_2 \in G$ from
\[
A c_1^{-1} \cap \{ g \in G : b_2^{-1} A \in \ultra{q} \}
\]
and not equal to $b_1$, choose $c_2 \in G$ from
\[
b_1^{-1} A \cap b_2^{-1} A \cap \{ g \in G : A g^{-1} \in \ultra{p} \}
\]
not equal to $c_1$ and so on.
We can choose at each step a never before chosen element of $G$ because all intersections belong to non-principal ultrafilters and are therefore infinite.
\end{proof}

The first step in the proof of \cref{thm_erdos_sumset_amenable} is the following reformulation, which involves multiplication by elements of $G$ from both the left and the right.
Because of this we need to work with two-sided \Folner{} sequences.
We would like to know whether \cref{thm_erdos_sumset_amenable} also holds for one-sided \Folner{} sequences.

\begin{theorem}
\label{thm_erdos_ultrafilters_group}
Let $G$ be a countable, amenable group and fix $A \subset G$.
If there exist a two-sided \Folner{} sequence $\Phi$ on $G$ and a non-principal ultrafilter $\ultra{p} \in \beta G$ such that $\dens_\Phi\big( Ag^{-1} \cap A \ultra{p}^{-1} )$ exists for all $g \in G$ and
\begin{equation}
\label{eqn_erdos_ultrafilters_group}
\lim_{g \to \ultra{p}} \dens_\Phi\big( Ag^{-1} \cap A\ultra{p}^{-1} ) > 0
\end{equation}
then there exist infinite sets $B,C$ such that $A\supset BC$.
\end{theorem}
\begin{proof}
Suppose that $\Phi$ and $\ultra{p}$ are as in the hypothesis with \eqref{eqn_erdos_ultrafilters_group} true.
Take $L = A \ultra{p}^{-1}$.
Then $g^{-1} A \in \ultra{p}$ for every $g \in L$.
We can find $\epsilon > 0$ such that
\[
\{ g \in G : \dens_\Phi(Ag^{-1} \cap L) > \epsilon \}
\]
belongs to $\ultra{p}$ and is therefore infinite.
It follows that
\[
\{ g \in G : \dens_\Phi( Ag^{-1} \cap L) > \epsilon \} \cap \bigcap_{h \in F} h^{-1} A
\]
is infinite for any finite set $F \subset L$.

Let $F_1 \subset F_2 \subset \cdots$ be an increasing exhaustion of $L$ by finite subsets.
Construct a sequence $n \mapsto e_n$ in $G$ of distinct elements such that
\[
e_n \in \{ g \in G : \dens_\Phi(Ag^{-1} \cap L) > \epsilon \} \cap \bigcap_{h \in F_n} h^{-1} A
\]
for each $n \in \N$.
This can be done because each of the sets above is infinite by hypothesis.

In particular $\dens_\Phi(Ae_n^{-1} \cap L) > \epsilon$ for all $n \in \N$.
The Bergelson intersectivity lemma (\cref{cor_bergelsonlemma}) implies that, for some subsequence $n \mapsto e_{\sigma(n)}$ of $e$ the intersection
\[
\Big( A e_{\sigma(1)}^{-1} \cap L \Big) \cap \cdots \cap \Big( Ae_{\sigma(n)}^{-1} \cap L \Big)
\]
is infinite for all $n \in \N$.

Choose $b_1 \in F_{\sigma(1)}$ and put $j_1 = 1$.
Choose $c_1 = e_{\sigma(1)}$.
Thus $c_1 \in b_1^{-1} A$.
Next choose $b_2 \in Ac_1^{-1} \cap L$ outside $F_{\sigma(1)}$ and let $j_2$ be minimal with $b_2 \in F_{\sigma(j_2)}$.
(In particular $b_2$ is not equal to $b_1$.)
Then choose $c_2 = e_{\sigma(j_2)} \in b_1^{-1} A \cap b_2^{-1} A$.
Continue this process inductively, choosing
\[
b_{n+1}
\in
A c_1^{-1} \cap \cdots \cap A c_n^{-1} \cap L
=
A e_{\sigma(j_1)}^{-1} \cap \cdots \cap A c_{\sigma(j_n)}^{-1} \cap L
\]
outside $F_{\sigma(j_n)}$ and choosing $j_{n+1}$ minimal with $b_{n+1} \in F_{\sigma(j_{n+1})}$ and then choosing
\[
c_{n+1} = e_{\sigma(j_{n+1})} \in b_1^{-1}A \cap \cdots \cap b_{n+1}^{-1} A
\]
which is distinct from $c_1,\dots,c_n$ because $e$ is injective.
Take $B = \{ b_n : n \in \N \}$ and $C = \{ c_n : n \in \N \}$ to conclude the proof.
\end{proof}

Our goal, given $A \subset G$ with positive upper density, is to find an ultrafilter $\ultra{p}$ and a two-sided \Folner{} sequence $\Phi$ satisfying \eqref{eqn_erdos_ultrafilters_group}.
To do this we work in the space
\[
\lp{2}(G,\Phi) = \{ f \colon G \to \C : \nbar f \nbar_\Phi < \infty \}
\]
where $\|f\|_{\Phi}$ is the \define{Besicovitch seminorm} of $f$ along a two-sided \Folner{} sequence $\Phi$ on $G$ defined as
\[
\nbar f \nbar_\Phi = \left( \limsup_{N \to \infty} \frac{1}{|\Phi_N|} \sum_{g \in \Phi_N} |f(g)|^2 \right)^{1/2}
\]
for all $f \colon G \to \C$.
Given $f,h \in \lp{2}(G,\Phi)$ write also
\[
\bilin{f}{h}_\Phi = \lim_{N \to \infty} \frac{1}{|\Phi_N|} \sum_{g \in \Phi_N} f(g) \overline{h(g)}
\]
whenever the limit exists.
Given a bounded function $f \colon G \to \C$ define, for all $g \in G$, the shift $\rmult^g f \colon G \to \C$ by $(\rmult^g f)(h) \coloneqq f(hg)$ for all $g \in G$ and, for all $\ultra{p}\in\beta G$, the function $\rmult^\ultra{p} f \colon G \to \C$ by $(\rmult^\ultra{p} f)(h) \coloneqq \lim_{g \to \ultra{p}} f(hg)$ for all $h \in G$.
One can check that the function $\rmult^\ultra{p}1_A$ is the indicator function of $A\ultra{p}^{-1}$.
Our ultimate goal is now reformulated in terms of $\lp{2}(G,\Phi)$ and $\rmult$ in the following theorem, which is analogous to \cref{thm_hilbertErdosNatural}.

\begin{theorem}
\label{thm_hilbertErdosGroup}
Let $G$ be a countable amenable group and fix $A \subset G$.
Let $\Phi$ be a two-sided \Folner{} sequence on $G$ such that $\dens_\Phi(A)$ exists.
For every $\epsilon > 0$ there exists a subsequence $\Psi$ of $\Phi$ and a non-principal ultrafilter $\ultra{p} \in \beta G$ such that $\bilin{\rmult^g 1_A}{\rmult^\ultra{p} 1_A}_\Psi$ exists for all $g \in G$ and
\begin{equation}
\label{eqn_hilbertErdosGroup}
\lim_{g \to \ultra{p}} \bilin{ \rmult^g 1_A }{ \rmult^\ultra{p} 1_A }_\Psi
\geq
\bilin{1}{1_A}_\Psi^2  - \epsilon
\end{equation}
holds.
\end{theorem}

As over $\N$ we will need to split $1_A$ into structured and pseudo-random components in two ways.
For the first we use finite dimensional representations to define an analogue of trigonometric polynomials.

\begin{definition}
By a \define{matrix coefficient} of a countable group $G$ we mean any map $a \colon G \to \C$ of the form $a(g) = \bilin{v}{M(g)w}$ for some homomorphism $M$ from $G$ to the unitary group $\unitary(n)$ over $\C^n$ and some vectors $v,w \in \C^n$ for some $n \in \N$.
A function $f \colon G \to \C$ is \define{Besicovitch almost periodic} along a two-sided \Folner{} sequence $\Phi$ on $G$ if, for every $\epsilon > 0$, one can find a matrix coefficient $a$ with $\nbar f - a \nbar_\Phi < \epsilon$.
\end{definition}

\begin{definition}
The set $\bes(G,\Phi)^\perp$ is defined to consist of those functions $f\in\lp{2}(G,\Phi)$ such that
\[
\lim_{N \to \infty} \frac{1}{|\Phi_N|} \sum_{g \in \Phi_N} f(g) a(g) = 0
\]
for all matrix coefficients $a$.
\end{definition}

Write $\bes(G,\Phi)$ for the set of functions $f$ in $\lp{2}(G,\Phi)$ that are Besicovitch almost periodic along $\Phi$.
We have the following splitting result.

\begin{theorem}
\label{thm_besSplittingGroup}
For every two-sided \Folner{} sequence $\Phi$ on $G$ and any $f \in \lp{2}(G,\Phi)$ there is a subsequence $\Psi$ of $\Phi$ and a function $f_\bes$ in $\lp{2}(G,\Psi)$ which is Besicovitch almost periodic along $\Psi$, and such that $f - f_\bes\in\bes(G,\Psi)^\perp$.
Moreover, if $f$ takes values in an interval $[a,b]\subset\R$ then so does $f_\bes$.
\end{theorem}

\begin{proof}
The definition of a projection family makes sense, and the proof of \cref{thm_general_dichotomy} goes through, without complication with $\N$ replaced by $G$.
It therefore suffices, in order to prove the result in question, to show that $\Phi \mapsto \bes(G,\Phi)$ is a projection family.

The only property that is not immediate is that the inner product $\langle a,b\rangle_\Phi$ exists whenever $a,b$ are matrix coefficients.
This follows from an application of the mean ergodic theorem; alternatively we provide the following short self contained proof.
Write $a(g) = \bilin{v}{M(g)w}$ and $b(g) = \bilin{r}{\tilde M(g) s}$ for homomorphisms $M \colon G \to \unitary(n)$ and $\tilde M \colon G \to \unitary(m)$ and appropriate vectors $r,s,u,v$.
Then $a(g) b(g)$ is a matrix coefficient for the tensor product representation $M \otimes \tilde M$ on $\C^{nm}$.

Now, if $a(g) = \bilin{v}{M(g) w}$ is any matrix coefficient the average
\[
\frac{1}{|\Phi_N|} \sum_{g \in \Phi_N} a(g)
=
\left\langle v,\frac{1}{|\Phi_N|} \sum_{g \in \Phi_N} M(g) w\right\rangle
\]
converges because, for all two-sided \Folner{} sequences $\Phi$ the sequence
\[
N \mapsto \frac{1}{|\Phi_N|} \sum_{g \in \Phi_N} \delta_{M(g)}
\]
of probability measures on $\unitary(n)$ converges in the weak topology to Haar measure on the closure of the image of $M$.
\end{proof}

The second splitting theorem is proved exactly as in \cref{sec_jdlgSplitting_for_N}.
We formulate here the appropriate generalizations of compact and weak mixing function.

\begin{definition}
A function $f \in \lp{2}(G,\Phi)$ is \define{compact} along $\Phi$ if, for every $\epsilon > 0$, one can find $F \subset G$ finite with $\min \{ \nbar \rmult^g f - \rmult^h f \nbar_\Phi : h \in F \} < \epsilon$ for all $g \in G$.
\end{definition}

\begin{definition}
A function $f \in \lp{2}(G,\Phi)$ is \define{weak mixing} along $\Phi$ if, for every bounded function $h \colon G \to \C$ and every subsequence $\Psi$ of $\Phi$ such that $\bilin{\rmult^g f}{h}_\Psi$ exists for all $g \in G$, the set $\{ g \in G : | \bilin{\rmult^g f}{h}_\Psi| > \epsilon \}$ has zero density with respect to every two-sided \Folner{} sequence on $G$.
\end{definition}

The proof of the following theorem is exactly as in \cref{sec_jdlgSplitting_for_N}.
For an appropriate version of the Jacobs--de Leeuw--Glicksberg splitting for unitary representations of groups see \cite[Chapter~16]{MR3410920} .

\begin{theorem}
\label{thm_jdlgSplittingGroup}
For every two-sided \Folner{} sequence $\Phi$ on $G$ and any $f \in \lp{2}(G,\Phi)$ there is a subsequence $\Psi$ of $\Phi$ and functions $f_\comp,f_\wm \in \lp{2}(G,\Psi)$ with $f_\comp$ compact along $\Psi$, $f_\wm$ weak mixing along $\Psi$, and $f = f_\comp + f_\wm$.
Moreover, if $f$ is real-valued and $a \le f \le b$ for some $a \le b$ then $f_\comp$ is also real valued and satisfies $a \le f_\comp \le b$.
\end{theorem}

The next ingredient in the proof of \cref{thm_hilbertErdosGroup} is an analogue of \cref{thm_choosingUltra}.
Its statement over $G$ and how it, together with \cref{thm_jdlgSplittingGroup} and \cref{thm_besSplittingGroup}, imply \cref{thm_hilbertErdosGroup}, is exactly the same as the proof of \cref{thm_hilbertErdosNatural} at the end of \cref{sec_outlineNatural}.
Its proof, also, is just as in \cref{subsec_choosingUltra} but using the following ingredients.

Definitions~\ref{def_essential_ultrafilter} and \ref{def_almost_everywhere} as well as Lemmas \ref{lemma_support} and \ref{lem_bohr_sets_have_positive_density} make sense in arbitrary countable groups.
The next three results -- versions of \cref{cor_compact_returns_Bohralternative}, \cref{thm_pointwise} and \cref{thm_beig} for countable, amenable groups -- fill the remaining gaps in the proof of \cref{thm_hilbertErdosGroup}.
First we recast \cref{def_bohr_set} for countable groups.

\begin{definition}
A \define{Bohr set} in a group $G$ is any set of the form $a^{-1}(U)$ where $a$ is a homomorphism from $G$ into a compact group $K$ and $U \subset K$ is a non-empty open set whose boundary has Haar measure $0$.
A Bohr set is a \define{Bohr$_0$ set} if $U$ contains the identity of $K$.
\end{definition}

For more details on Bohr sets in amenable groups see \cite[Subsection 1.3]{MR2565535}.

\begin{lemma}
For every $f \in \lp{2}(G,\Phi)$ that is compact along $\Phi$ and every $\epsilon > 0$ the set $\{ g \in G : \nbar \rmult^g f - f \nbar_\Phi < \epsilon \}$ contains a Bohr$_0$ set.
\end{lemma}
\begin{proof}
Since $f$ is compact along $\Phi$ the function $\phi \colon g \mapsto \nbar \rmult^g f - f \nbar_\Phi$ has the property that the set $\{ \rmult^h \phi : h \in G \}$ has compact closure with respect to the uniform norm on bounded functions $G \to \C$.
By \cite[Remark~9.8]{MR513591} there is a compact topological group $K$ and a continuous homomorphism $\xi \colon G \to K$ and a continuous function $\psi \colon K \to \C$ such that $\phi(g) = \psi(\xi(g))$.
Therefore the set $\{ g \in G : \nbar \rmult^g f - f \nbar_\Phi < \epsilon \}$ contains a Bohr$_0$ set.
\end{proof}

\begin{theorem}
If $h \colon G \to \C$ is bounded and Besicovitch along $\Phi$ then there is a subsequence $\Psi$ of $\Phi$ such that $\nbar \rmult^\ultra{p} h \nbar_\Psi = \nbar h \nbar_\Psi$ for $\Psi$ almost all $\ultra{p}$.
\end{theorem}
\begin{proof}
The proof is unchanged from the $\N$ case, except that we need to verify $\nbar \rmult^\ultra{p} a \nbar_\Psi = \nbar a \nbar_\Psi$ for all ultrafilters $\ultra{p}$, all two-sided \Folner{} sequences $\Psi$ and all matrix coefficients $a \colon G \to \C$.
Fix $A \colon G \to \umult(n)$ and $v,w \in \C^n$ with $a(g) = \bilin{v}{A(g)w}$ for all $g \in G$.
Let $K$ be the closure of the image of $A$ in $\umult(n)$ and let $\haar$ be its normalized Haar measure.
Writing $\psi(k) = \bilin{v}{kw}$ for all $k \in K$ we have, as in the proof of \cref{thm_besSplittingGroup}, that
\[
\nbar a \nbar_\Psi^2 = \int |\psi|^2 \d \haar
\]
for all two-sided \Folner{} sequences $\Psi$.
Since
\[
(\rmult^\ultra{p} a)(h)
=
\lim_{g \to \ultra{p}} \bilin{v}{A(h) A(g)w}
=
\bilin{v}{A(h) \ell w}
\]
for some $\ell \in K$ we have
\[
\nbar \rmult^\ultra{p} a \nbar_\Psi^2
=
\int |\psi(k\ell)|^2 \d \haar(k)
=
\int |\psi(k)|^2 \d \haar(k)
=
\nbar a \nbar_\Psi^2
\]
by invariance of Haar measure as desired.
\end{proof}

The last theorem -- a version of \cref{thm_beig} for countable, amenable groups -- is proved exactly as in \cref{sec_beig}.

\begin{theorem}
Suppose $f \colon G \to \R$ is a bounded function that is orthogonal to $\bes(G,\Psi)$.
Then for every non-empty Bohr set $B \subset G$ and every bounded function $h \colon G \to \R$ the set
\[
\left\{ \ultra{p} \in \ess(\Phi) : B \in \ultra{p} \textup{ and } \limsup_{N \to \infty} \frac{1}{|\Psi_N|} \sum_{g \in \Psi_N} h(g) \, (\rmult^\ultra{p} f)(g) \ge 0 \right\}
\]
has positive measure with respect to every $\mu \in \M{\Psi}$.
\end{theorem}

\section{Open questions}
\label{sec_questions}

Two natural questions, which arise from questions asked by \Erdos{} in \cite[Section 6]{Erdos77} and \cite[p.\ 105]{Erdos80}, are as follows.

\begin{question}
\label{q_no}
Does every set $A \subset \N$ satisfying
\[
\limsup_{N \to \infty} \frac{|A \cap \{1,\dots,N\}|}{N} > 0
\]
contain a set of the form $t + B + B$ where $t \in \N$ and $B\subset\N$ is infinite?
\end{question}


\begin{question}
\label{q_yes}
Does every set $A \subset \N$ satisfying
\[
\limsup_{N \to \infty} \frac{|A \cap \{1,\dots,N\}|}{N} > 0
\]
contain a set of the form $t + (B \oplus B) $ where $t \in \N$, $B\subset\N$ is infinite, and $B\oplus B \coloneqq \{b_1+b_2: b_1,b_2\in B,~ b_1\neq b_2\}$?
\end{question}

It was pointed out to us by Steven Leth that there exists a set of positive upper density that does not contain any set of the form $B+B+t$ for $t\in\N$ and infinite $B\subset\N$. In particular, the answer to \cref{q_no} is negative. An example of such a set is $A=\bigcup_{n=1}^\infty \left[4^n, \tfrac{3 }{2} 4^n \right]$.

We do not know the answer to \cref{q_yes}.
An ultrafilter reformulation of this question was obtained by Hindman in \cite[Section 11]{MR0564927}. We also refer the reader to another paper of Hindman \cite{MR0677568} which treats this question.
Note that an affirmative answer to \cref{q_yes} implies \cref{conjecture_erdossumset}.

\begin{question}
Suppose $A\subset\N$ has positive upper density. Do there exist infinite sets $B,C,D\subset\N$ such that the sum $B+C+D$ is contained in $A$?
Is it true that for every $k\in\N$ there exist infinite sets $B_1,\dots,B_k\subset\N$ such that $B_1+\cdots+B_k\subset A$?
\end{question}



The Green--Tao theorem on arithmetic progressions~\cite{MR2415379} gives a version of \Szemeredi{}'s theorem in the primes.
It is natural to ask (cf.\ \cite{MR1091350}) whether a version of the \Erdos{} sumset conjecture holds for the primes.

\begin{question}
\label{question_primes}
Let $\mathbb{P}$ denote the set of prime numbers. Are there infinite sets $B,C\subset\N$ such that $B+C\subset \mathbb{P}$?
\end{question}

A positive answer to \cref{question_primes}, conditional on the Hardy-Littlewood prime tuples conjecture, was obtained by Granville~\cite{MR1091350}.
(The authors thank Karl Mahlburg for this reference.)

Lastly we pose a more open-ended question which was asked by Jon Chaika.
\begin{question}Is there a version of \cref{thm_erdosNatural} over $\R$ or more general locally compact topological groups?
\end{question}

\printbibliography

\end{document}